\documentclass[a4paper]{article}

\usepackage{amsfonts}
\usepackage{amsmath}
\usepackage{mathrsfs}
\usepackage{amssymb,color}

\usepackage[latin1]{inputenc}
\usepackage{amsthm}
\usepackage{graphicx}
\usepackage{enumerate}
\newtheorem{theorem}{Theorem}[section]

\newtheorem{definition}[theorem]{Definition}

\newtheorem{remark}[theorem]{Remark}
\newtheorem{corollary}[theorem]{Corollary}

\begin{document}
	\title{\textbf{Existence of solutions for $n^\mathrm{th}$-order nonlinear differential boundary value problems by means of new fixed point theorems}}
	\author{Alberto Cabada and  Lorena Saavedra\\Instituto de Matem\'aticas,\\ Facultade de Matem\'aticas,\\Universidade de Santiago de Compostela,\\ Santiago de Compostela, Galicia,
		Spain\\
		alberto.cabada@usc.es, lorena.saavedra@usc.es
	} 
	\maketitle
	\begin{abstract}
	This paper is devoted to prove the existence of one or multiple solutions of a wide range of nonlinear differential boundary value problems.
	
	To this end, we obtain some new fixed point theorems for a class of integral operators. We follow the well-known Krasnoselski\u{\i}'s fixed point Theorem together with two fixed point results of  Leggett-Williams type.	
	After obtaining a general existence result for a one parameter family of nonlinear differential equations, are proved, as particular cases, existence results for second and fourth order nonlinear boundary value problems.
		
	\end{abstract}

	\section{Introduction}
	
	The use of different kind of fixed point theorems has been shown as a very useful tool to obtain the existence of solutions of nonlinear boundary value problems.

For instance, in \cite{CaCiIn} it is proved that the following second order system
\[\begin{cases}
u_1''(t)+\lambda_1\,f_1(u_1(t),u_2(t))=0\,,&t\in(0,1)\,,\\
u_2''(t)+\lambda_2\,f_2(u_1(t),u_2(t))=0\,,&t\in(0,1)\,,\\
u_1'(0)=u_1(1)+u_1'(1)=0\,,\\
u_2'(0)=u_2(1)+\varepsilon\,u_2(\eta)=0\,,&\eta\,,\ \varepsilon\in(0,1)\,,
\end{cases}\]
has a solution for every $\lambda_1$, $\lambda_2>0$, by applying some previously obtained fixed point results on a related system of integral operators.

In \cite{FrInPe}, under suitable conditions for the functions $f$, $g$ and $\lambda\in \mathbb{R}$, it is proved the existence of at least two strictly positive solutions of the second order boundary value problem
\[\begin{cases}
u''(t)+\lambda\,g(t)\,f(u(t))=0\,,&t\in(0,1)\,,\\
u'(0)=\sigma\,u'(1)+u(\eta)=0\,,&\eta\in[0,1]\,,
\end{cases}\]	
which models the behavior of a thermostat. Again, the existence result is obtained by applying a fixed point theorem.

In \cite{CaPrSaTe}, it is studied a fourth order boundary value problem coupled with the cantilever beam boundary conditions:
\[\begin{cases}
u^{(4)}(t)=f(t,u(t))\,,&t\in(0,1)\,,\\
u(0)=u'(0)=u''(1)=u'''(1)=0\,.
\end{cases}\]

In such a work, two combined techniques on the existence of solutions are used: the critical and the fixed point theory.

In \cite{AnAvHen}, it is obtained a result in the line of the Leggett-Williams fixed point theorem (see \cite{LeWi}), that guarantees the existence of at least  a positive  fixed point on different sets defined by means of suitable functionals. As a direct application of this result, it is proved the existence of solution of the following second order nonlinear differential boundary value problem:
\[\begin{cases}
u''(t)+f(u(t))=0\,,& t\in(0,1)\,,\\
u(0)=u'(1)=0\,.&
\end{cases}\]

Moreover, in \cite{AvAnHen}, by means of a new fixed point theorem proved in that reference, a different existence result for this problem has been obtained.

In \cite{ZhSu} it is obtained an extension of the Leggett-Williams fixed point theorem given in \cite{LeWi} and, as an application, it is obtained a result which ensures the existence of multiple solutions of the following third order boundary value problem:
\[\begin{cases}
u'''(t)+f(t,u(t),u'(t),u''(t))=0\,,&t\in(0,1)\,,\\
u(0)=\sum_{i=1}^{m-2}k_i\,u(\xi_i)\,,\  u'(0)=u'(1)=0\,,
\end{cases}\]
where $0<\xi_1<\xi_2<\cdots<\xi_{m-2}<1$, $k_i\in\mathbb{R}^+$ for $i=1,\dots,m-2$ and $0<\sum_{i=1}^{m-2}k_i<1$.

In \cite{AvElHe}, as an application of the Leggett-Williams fixed point Theorem, it is proved the existence of at least one solution for the problem:
\[\begin{cases}
u^{(4)}(t)=f(u(t))\,,&t\in(0,1)\,,\\
u(0)=u'(0)=u(1)=u'(1)=0\,,
\end{cases}\]
under suitable conditions of $f$.

In this paper, as an application of the results here presented, on Section 5 we prove existence and multiplicity results for a fourth order problem with the same clamped beam boundary conditions. However, in our case, we consider a non-autonomous nonlinear part $f(t,u)$. Moreover, we prove, under suitable conditions of $f$, the existence of at least two or three solutions of the considered problem.

In addition, in that section we study a family of second order problems couped with the Dirichlet boundary conditions:
\[\begin{cases}
u''(t)+B\,u'(t)+f(t,u(t))=0\,,&t\in(0,1)\,,\\
u(0)=u(1)=0\,,
\end{cases}\]
where $B\in\mathbb{R}$.

In fact, such results are included in a more general framework, which is delivered to ensure general existence and multiplicity results of a one parameter family of $n^{\mathrm{th}}$-order boundary value problems given by
\begin{equation}
\label{Ec::Dif} (-1)^{n-k} T_n[M]\,u(t)=f(t,u(t))\,,\quad t\in I\equiv [a,b]\,,
\end{equation}
coupled with the so-called $(k,n-k)$ boundary conditions
\begin{equation}
\label{Ec::BC} \begin{split}
u(a)=u'(a)=\cdots=u^{(k-1)}(a)&=0\,,\\
u(b)=u'(b)=\cdots=u^{(n-k-1)}(b)&=0\,,
\end{split}
\end{equation}
where $k\in\{1,\dots,n-1\}$ and
\begin{equation}\label{Ec::Tn}T_n[M]\,u(t)=u^{(n)}(t)+p_1(t)\,u^{(n-1)}(t)+\cdots+(p_n(t)+M)\,u(t)\,,\quad t\in I\,,\end{equation}

 It is well-known that the solutions of problem \eqref{Ec::Dif}-\eqref{Ec::BC} are given as the fixed points of the integral operator
 \begin{equation}\label{Ec::OpIn}\mathcal L_1 \,u(t)=\int_a^b(-1)^{n-k}g_M(t,s)\,f(s,u(s))\,ds\,,\end{equation}
 where $g_M(t,s)$ is the associated Green's function of the operator $T_n[M]$ in the space related to the given boundary conditions.
 
 Thus, in order to find the fixed points of operator ${\mathcal L_1}$, we previously study, in Sections 3 and 4, the existence of one or multiple positive fixed points of an integral operator defined as follows:
	\begin{equation}\label{Ec::OpInt}\begin{array}{rcl}
\mathcal L\colon C(I)&\longrightarrow&C(I)\\
	u&\longmapsto&\mathcal L\,u(t):=\int_a^bG(t,s)\,f(s,u(s))\,ds\,,\end{array}\end{equation}
	where 
		\[\begin{array}{rcl}
	f\colon I\times \mathbb{R}&\longrightarrow&\mathbb{R}\\
	(t,u)&\longmapsto&f(t,u)\,,\end{array}\]
	is a continuous function satisfying that $f(t,u)\geqslant 0$ for all $(t,u)\in I\times \mathbb{R}^+$ and 
		\[\begin{array}{rcl}
	G\colon I\times I&\longrightarrow&\mathbb{R}\\
	(t,s)&\longmapsto&G(t,s)\,,\end{array}\]
	is an integral kernel.
	
	To this end, we impose the following hypothesis on the kernel $G$:
	\begin{itemize}
		\item [$\left(Pg_1\right)$] There exist $\varPhi$, $k_1$ and $k_2$ continuous functions on $ I$ such that $\varPhi(s)>0$ for all $s\in(a,b)$, $0<k_1(t)\leqslant k_2(t)$ for all $t\in(a,b)$ and
		\[\varPhi(s)\,k_1(t)\leqslant G(t,s)\leqslant \varPhi(s)\,k_2(t)\,,\quad \forall(t,s)\in I\times I\,.\]
	\end{itemize}

It is important to note that  (see for instance\cite{Cab, Cop, karlin, Lan} and references therein) there are obtained several results for the existence of one or multiple fixed points of integral operators by imposing similar conditions to property $(Pg_1)$ to the kernel $G$.

In this paper, we use the well-known Krasnoselski\u{\i}'s fixed point Theorem, collected in \cite{Kra}. Moreover, in order to obtain the existence of two or three fixed points we use two results due to Avery and Henderson, \cite{AvHe}, and Avery, \cite{Av}, respectively. The structure what  we follow is the one given in \cite{AndHof}, where these results are used to prove the existence of one or multiple solutions for a problem on time scales.

Our final purpose is to prove the existence of solutions of the problem \eqref{Ec::Dif}-\eqref{Ec::BC}. Thus, we need to guarantee that the related Green's function satisfies the property $(Pg_1)$.
In \cite{CabSaa}, by means of spectral theory, it has been obtained a characterization of the values of the parameter $M$ for which the Green's function, related to the $(k,n-k)$ boundary conditions, has constant sign. A fact which, in this case, also implies that such a function verifies the property $(Pg_1)$. Hence, as an application of the previously obtained fixed point theorems, we will prove the existence of one or multiple positive solutions for $(k,n-k)$ nonlinear boundary value problems.

In addition, this property can be extended for many different boundary conditions. For instance, in \cite{CabSaa2}, a fourth order problem coupled with the  simply supported boundary conditions, $u(a)=u(b)=u''(a)=u''(b)=0$, is studied. Thus, the results here shown can also be applied in such a case.

This paper is structured as follows: in  next section, we  describe the studied problem and show some preliminary results which are used along the paper. In Section 3, we obtain some results that ensure the existence of one or multiple fixed points by using the Krasnoseslki\u{\i}'s fixed point Theorem given in \cite{Kra}. Then, in Section 4, following the results of  Avery and Henderson, \cite{AvHe}, and Avery, \cite{Av}, we obtain the existence of at least two or three fixed points, respectively.  Finally, in Section 5,  as an application of these fixed point theorems, we deduce the existence and multiplicity of solutions of problem \eqref{Ec::Dif}-\eqref{Ec::BC}. Moreover, as particular cases, second and fourth order boundary value problems are considered.

	\section{Description of the problem and some previous fixed point existence results}
	The aim of this paper is to study the existence of some fixed points of the integral operator described in \eqref{Ec::OpInt} in an appropriate cone.
	
	First, let us define the concept of cone.
	
	\begin{definition}
		Let $\mathcal{B}$ be a real Banach space. A nonempty closed convex set $\mathcal P\subset \mathcal B$ is called a cone if it satisfies the following two conditions:
		\begin{enumerate}
			\item $\lambda\,x\in\mathcal P$ for all $x\in\mathcal P$ and $\lambda \geqslant 0$.
			\item If $x\in\mathcal{P}$ and $-x\in \mathcal P$, then $x=0$.
		\end{enumerate}
	\end{definition}

In the sequel, we describe the cone where the fixed points are found, as well as some constants which are used along the paper. 

Let us consider a subinterval $I_1=[a_1,b_1]\subset  I$ such that $k_1(t)>0$ for all $t\in I_1$. Then, we denote:
\[K_1=\max_{t\in I}k_1(t)>0\,,\quad m_1=\min_{t\in I_1}k_1(t)>0\,,\quad K_2=\max_{t\in I}k_2(t)>0\,.\]

Consider the Banach space $\mathcal B=C(I)$ coupled with the norm 
\[\|u\|_\infty=\max_{t\in I}|u(t)|\,,\]
and the cone
\[\mathcal P=\left\lbrace u\in \mathcal{B}\ \mid u(t)\geqslant \frac{k_1(t)}{K_2}\,\|u\|_{\infty}\,,\quad t\in I\right\rbrace \,.\]

In the sequel, to make the paper more readable, we show some preliminary results which will be used along the paper.
	
First, let us consider the Krasnoselski\u{\i}'s fixed point Theorem, \cite{Kra}:

\begin{theorem}
	\label{T::1}
	Let $\mathcal{B}$ be a Banach space, $\mathcal P\subset \mathcal B$ be a cone, and suppose that $\Omega_1$, $\Omega_2$ are bounded open balls of $\mathcal{B}$ centered at the origin, with $\overline{\Omega}_1\subset\Omega_2$. Suppose further that $\mathcal L\colon \mathcal P\cap(\overline{\Omega}_2\backslash \Omega_1)\rightarrow \mathcal P$ is a completely continuous operator such that either
	\[\left\| \mathcal Lu\right\| \leqslant \left\| u\right\| \,,\quad u\in \mathcal P\cap \partial \,\Omega_1\,,\quad \text{and}\quad \left\| \mathcal Lu\right\| \geqslant \left\| u\right\| \,,\quad u\in \mathcal P\cap \partial\, \Omega_2\,,\]
	or
	\[\left\| \mathcal Lu\right\| \geqslant \left\| u\right\| \,,\quad u\in \mathcal P\cap \partial \,\Omega_1\,,\quad \text{and}\quad \left\| \mathcal Lu\right\| \leqslant \left\| u\right\| \,,\quad u\in \mathcal P\cap \partial\,
	 \Omega_2\,,\]
	 holds. Then $\mathcal L$ has a fixed point in $\mathcal P\cap (\overline{\Omega}_2\backslash\Omega_1)$.
\end{theorem}

\begin{definition}
A map $\alpha$ is said to be a nonnegative continuous concave functional on a cone $\mathcal P$ of a real Banach space $\mathcal B$ if $\alpha\,\colon\,\mathcal P\,\rightarrow\,[0,+\infty)$ is continuous and
\[\alpha(t\,x+(1-t)\,y)\geqslant t\,\alpha(x)+(1-t)\,\alpha(y)\,,\quad \forall \,x\,,\ y\in\mathcal P\,,\quad t\in[0,1]\,.\]

Similarly, a map $\beta$ is said to be a nonnegative continuous convex functional on a cone $\mathcal P$ of a real Banach space $\mathcal B$ if $\beta\,\colon\,\mathcal P\,\rightarrow\,[0,+\infty)$ is continuous and
\[\beta(t\,x+(1-t)\,y)\leqslant t\,\beta(x)+(1-t)\,\beta(y)\,,\quad \forall \,x\,,\ y\in\mathcal P\,,\quad t\in[0,1]\,.\]
\end{definition}

Now, let us consider $\beta$, $\gamma$ and $\theta$, nonnegative continuous convex functionals on the cone $\mathcal P$, and $\alpha$ and $\psi$, nonnegative concave functionals on $\mathcal P$.

For nonnegative real numbers $d$, $p$ and $q$, we define the following subspaces of the cone $\mathcal P$:
\[\begin{split}
\mathcal P(\gamma, r)&=\left\lbrace u\in \mathcal P\ \mid\ \gamma(u)<r\right\rbrace\,,\\
\mathcal P(\gamma,\alpha,p,r)&=\left\lbrace u\in\mathcal P\ \mid\ p\leqslant \alpha(u)\,,\ \gamma(u)\leqslant r\right\rbrace \,,\\
\mathcal Q(\gamma,\beta,d,r)&=\left\lbrace u\in\mathcal P\ \mid\ \beta(u)\leqslant d\,,\ \gamma(u)\leqslant r\right\rbrace \,,\\
\mathcal P(\gamma,\theta,\alpha,p,q,r)&=\left\lbrace u\in\mathcal P\ \mid\ p\leqslant \alpha(u)\,,\ \theta(u)\leqslant q\,,\ \gamma(u)\leqslant r\right\rbrace \,.
\end{split}\]

In the sequel we introduce a result, proved in \cite{AvHe}, which ensures the existence of two fixed points on the cone $\mathcal P$.

\begin{theorem}\label{T::1.1}
	Let $\mathcal P$ be a cone in a real Banach space $\mathcal B$. Let $\alpha$ and $\gamma$ be increasing and nonnegative continuous functionals on $\mathcal P$. Let $\theta$ be a nonnegative continuous functional on $\mathcal P$ with $\theta(0)=0$ such that, for some positive constants $r$ and $M$,
	\[\alpha(u)\leqslant \theta(u)\leqslant \gamma(u)\quad \text{and}\quad \|u\|\leqslant M\,\alpha(u)\,,\quad \text{for all } u\in \overline{\mathcal P(\alpha,r)}\,.\]
	
Assume that there exist two positive numbers $p$ and $q$ with $p<q<r$ such that
	\[\theta(\lambda\,u)\leqslant \lambda\,\theta(u)\,,\quad \text{for all } 0\leqslant \lambda\leqslant 1\quad\text{and}\quad u\in\partial \mathcal P(\theta,q)\,.\]
	
	Suppose that $\mathcal L\colon  \overline{\mathcal P(\alpha,r)}\rightarrow \mathcal P$ is a completely continuous operator satisfying
	\begin{enumerate}[i)]
		\item $\alpha(\mathcal L\,u)>r$ for all $u\in\partial \mathcal P(\alpha,r)$,
		\item $\theta(\mathcal L\,u)<q$ for all $u\in \partial \mathcal P(\theta,q)$,
		\item $\mathcal P (\gamma,p)\neq \emptyset$ and $\gamma(\mathcal{L}\,u)>p$ for all $u\in\partial\mathcal P(\gamma,p)$.
	\end{enumerate}

Then, $\mathcal L$ has at least two fixed points $u_1$ and $u_2$ such that
\[p<\gamma(u_1)\,, \quad\text{with}\quad \theta(u_1)<q\,,\]
and
\[q<\theta(u_2)\,, \quad\text{with}\quad \alpha(u_2)<r\,.\]
\end{theorem}

\vspace{0.5cm}
Finally, we introduce a result, proved in \cite{Av}, which warrants the existence of three fixed points of the operator $\mathcal L$ on the cone $\mathcal P$.

\begin{theorem}
	\label{T::1.2}
		Let $\mathcal P$ be a cone in a real Banach space $\mathcal B$, and let $r$ and $M$ be positive numbers. Assume that $\alpha$ and $\psi$ are nonnegative, continuous and  concave functionals on $\mathcal P$, and $\gamma$, $\beta$ and $\theta$ are nonnegative, continuous and convex functionals on $\mathcal P$ with
		\[\alpha(u)\leqslant \beta(u)\quad \text{and}\quad \|u\|\leqslant M\,\gamma(u)\,,\quad\forall\, u\in \overline{\mathcal P (\gamma,r)}\,.\]
	Suppose that $\mathcal L\colon  \overline{\mathcal P(\gamma,r)}\rightarrow  \overline{\mathcal P(\gamma,r)}$ is a completely continuous operator and there exist nonnegative numbers $h$, $d$, $p$, $q$, with $0<d<p$, such that:
	\begin{enumerate}[a)]
		\item $\left\lbrace u\in \mathcal P(\gamma,\theta,\alpha,p,q,r)\ \mid\ \alpha(u)>p\right\rbrace\neq \emptyset$ and $\alpha(\mathcal L\,u)>p$ for $u\in \mathcal P(\gamma,\theta,\alpha,p,q,r)$,
		\item $\left\lbrace u\in \mathcal P(\gamma,\beta,\psi,h,d,r)\ \mid\ \beta<d\right\rbrace\neq \emptyset$ and $\beta(\mathcal L\,u)<d$ for $u\in \mathcal P(\gamma,\beta,\psi,h,d,r)$,
		\item $\alpha(\mathcal L\,u)>p$ for $u\in \mathcal P(\gamma,\alpha,p,r)$ with $\theta(\mathcal L\,u)>q$,
		\item $\beta(\mathcal L\,u)<d$ for $u\in \mathcal Q(\gamma,\beta,d,r)$ with $\Psi(\mathcal L\,u)<h$.
	\end{enumerate}	

Then $\mathcal L$ has at least three fixed points $u_1$, $u_2$, $u_3\in\overline{\mathcal P (\gamma,r)}$ such that
\[\beta(u_1)<d\,,\quad p<\alpha(u_2)\quad \text{and}\quad \text d<\beta (u_3)\quad \text{with}\quad \alpha(u_3)<p\,.\]
\end{theorem}

\section{Existence of fixed points by means of Krasnoselski\u{\i} fixed point Theorem}

The aim of this section consists on ensuring the existence of at least a fixed point of operator $\mathcal{L}$, defined in \eqref{Ec::OpInt}. Such an existence will follow as an application of Theorem \ref{T::1}.

Let us consider the following conditions on $f$:
\begin{itemize}
	\item[$(H_1)$] There exists $p>0$ such that
	\[f(t,u)\leqslant \dfrac{p}{K_2\,\int_a^b\varPhi(s)\,ds}\,,\quad \forall\, t\in I\,,\ \forall\, u\in[0,p]\,.\]
	\item[$(H_2)$] There exists $q>0$ such that 
		\[f(t,u)\geqslant \dfrac{K_2\,u}{K_1\,\int_{a_1}^{b_1}k_1(s)\,\varPhi(s)\,ds}\,,\quad \forall\, t\in I_1\,,\ \forall\, u\in\left[ \frac{m_1}{K_2}\,q,q\right] \,.\]
\end{itemize}

Following the same steps as in \cite[Theorem 2.3]{AndHof}, we prove the following result

\begin{theorem}
	\label{T::2}Suppose that there exist two positive numbers $p\neq q$ such that condition $(H_1)$ is satisfied with respect to $p$ and condition $(H_2)$ is satisfied with respect to $q$. Then, provided that the integral kernel $G$ satisfies $(Pg_1)$, operator $\mathcal{L}$, defined in \eqref{Ec::OpInt}, has a fixed point, $u\in \mathcal P$, such that $\|u\|_\infty$ lies between $p$ and $q$.
\end{theorem}

\begin{proof}
	First let us see that $\mathcal L(\mathcal P)\subset \mathcal P$. 
	
	Let $u\in \mathcal P$, we have:
	\begin{align}
	\nonumber
	\mathcal L\,u(t):=& \int_a^b G(t,s)\,f(s,u(s))\,ds\geqslant \int_a^b k_1(t)\,\varPhi(s)\,f(s,u(s))\,ds\\\nonumber=&\dfrac{k_1(t)}{K_2}
\int_a^bK_2\,\varPhi(s)\,f(s,u(s))\,ds\geqslant \dfrac{k_1(t)}{K_2}
\int_a^b\sup_{t\in I}\left\lbrace G(t,s)\right\rbrace \,f(s,u(s))\,ds\\\nonumber\geq& \dfrac{k_1(t)}{K_2}
\sup_{t\in I}\left\lbrace \int_a^bG(t,s)\,f(s,u(s))\,ds\right\rbrace =\dfrac{k_1(t)}{K_2}
\left\| \mathcal L\,u\right\|_\infty\,,\quad \forall t\in I\,.\end{align}

Thus,  since $f$ is a continuous function on $I\times \mathbb{R}$, $\mathcal L\,u\in \mathcal P$ and $\mathcal L\colon \mathcal P\rightarrow\mathcal P$ is a completely continuous operator.

Now, let us define the open balls centered at the origin as follows:
\[\Omega_p=\left\lbrace u\in C(I)\ \mid\ \|u\|_\infty<p\right\rbrace\quad \text{and}\quad \Omega_q=\left\lbrace u\in C(I)\ \mid\ \|u\|_\infty<q\right\rbrace\,. \]

From $(Pg_1)$ and the positiveness of $f$ for all $u\in \mathcal P$, we have that the following inequality is fulfilled.
\begin{equation}\label{Ec::Des1}\|\mathcal L\,u\|_\infty =\sup_{t\in I}\left\lbrace \int_a^b G(t,s)\,f(s,u(s))\,ds\right\rbrace \leqslant K_2\,\int_a^b\varPhi(s)\,f(s,u(s))\,ds\,.\end{equation}

On the other hand, for $u\in \mathcal P\cap \partial\Omega_p$, we have that $\|u\|_\infty=p$, so, from \eqref{Ec::Des1} and $(H_1)$, we have:
\[\|\mathcal L\,u\|_\infty\leqslant K_2\,\int_a^b\varPhi(s)\,f(s,u(s))\,ds\leqslant K_2\int_a^b\varPhi(s)\dfrac{p}{K_2\,\int_a^b\varPhi(s)\,ds}\,ds=p=\|u\|_\infty\,.\]

Thus, $\|\mathcal L\,u\|_\infty\leqslant \|u\|_\infty$ for all $u\in \mathcal P\cap\partial \Omega_p$.

Now, using  $(Pg_1)$ again, we have, for all $u\in \mathcal P$:
{\small \begin{equation}\label{Ec::Des2}\begin{split}\|\mathcal L\,u \|_\infty =&\sup_{t\in I}\left\lbrace \int_a^b G(t,s)\,f(s,u(s))\,ds\right\rbrace \geqslant\sup_{t\in I} \left\lbrace \int_a^bk_1(t)\,\varPhi(s)\,f(s,u(s))\,ds\right\rbrace \\=&K_1\int_a^b\varPhi(s)\,f(s,u(s))\,ds \,.\end{split}\end{equation}}

Let $v\in \mathcal P\cap \partial \Omega_q$, then \[\min_{t\in I_1}v(t)\geqslant \min_{t\in I_1}\dfrac{k_1(t)}{K_2}\|v\|_\infty=\dfrac{m_1}{K_2}\,q\,.\]

 So, from \eqref{Ec::Des2} and $(H_2)$:
\[\begin{split}\|\mathcal L\,v \|_\infty \geq& K_1\int_a^b\varPhi(s)\,f(s,v(s))\,ds\geqslant K_1\int_{a_1}^{b_1}\varPhi(s)\,f(s,v(s))\,ds\\\geqslant &K_1\,\int_{a_1}^{b_1}\varPhi(s)\,\dfrac{K_2\,v(s)}{K_1\,\int_{a_1}^{b_1}k_1(s)\,\varPhi(s)\,ds}\,ds\\\geqslant &\int_{a_1}^{b_1}\varPhi(s)\,\dfrac{K_2 \dfrac{k_1(s)}{K_2}\,\|v\|_\infty}{\int_{a_1}^{b_1}k_1(s)\,\varPhi(s)\,ds}\,ds=\|v\|_\infty=q\,.\end{split}\]

Hence, $\|\mathcal L\,v\|_\infty\geqslant \|v\|_\infty$ for all $v\in \mathcal P\cap\partial\Omega_q$.

Then, from Theorem \ref{T::1}, we conclude that $\mathcal L$ has a fixed point in $\mathcal P$ such that $\|u\|_\infty$ lies between $p$ and $q$.
\end{proof}

Define
\[\begin{split}f_0^+(t):=\limsup_{u\rightarrow0^+}\dfrac{f(t,u)}{u}\,,\quad f_\infty^+(t):=\limsup_{u\rightarrow\infty}\dfrac{f(t,u)}{u}\,,\\
f_0^-(t):=\liminf_{u\rightarrow0^+}\dfrac{f(t,u)}{u}\,,\quad f_\infty^-(t):=\liminf_{u\rightarrow\infty}\dfrac{f(t,u)}{u}\,.\end{split}\]
 
 Analogously to \cite[Corollary 2.4]{AndHof}, we have
 \begin{corollary}\label{Cor::2.4}
 	If $G$ satisfies $(Pg_1)$, then operator $\mathcal L$, defined in \eqref{Ec::OpInt},  has a fixed point in $\mathcal P$, provided that one of two following conditions hold 
 	\begin{itemize}
 		\item [$(H_3)$] $f_0^+(t)<\frac{1}{K_2\,\int_a^b\varPhi(s)\,ds}$ for $t\in I$ and $f_\infty^-(t)>\frac{K_2^2}{K_1\,m_1\int_{a_1}^{b_1}k_1(s)\,\varPhi(s)\,ds}$ for $t\in I_1$,
 		\item [$(H_4)$] $f_\infty^+(t)<\frac{1}{K_2\,\int_a^b\varPhi(s)\,ds}$ for $t\in I$ and $f_0^-(t)>\frac{K_2^2}{K_1\,m_1\int_{a_1}^{b_1}k_1(s)\,\varPhi(s)\,ds}$ for $t\in I_1$.
 	\end{itemize}
 In particular, there is one fixed point if $f_0^+=0$ and $f_\infty^-=\infty$ ($f_\infty^+=0$ and $f_0^-=\infty$).
 \end{corollary}

\begin{proof}
First, let us assume that $(H_3)$ is fulfilled.

Then, there exists $p>0$,  small enough, and $q>0$, big enough, such that:
\begin{itemize}
	\item $\dfrac{f(t,u)}{u}\leqslant \dfrac{1}{K_2\,\int_a^b\varPhi(s)\,ds}$ for all $t\in  I$ and $0< u\leqslant p$.
	\item $\dfrac{f(t,u)}{u}\geqslant \dfrac{K_2^2}{K_1\,m_1\int_{a_1}^{b_1}k_1(s)\,\varPhi(s)\,ds}$ for all $t\in I_1$ and $u\geqslant q$.
\end{itemize}
	
Thus, $(H_1)$ and $(H_2)$ are fulfilled for $p>0$ and $\frac{K_2}{m_1}q>0$, respectively. So, from Theorem \ref{T::2}, $\mathcal{L}$ has a fixed point on $\mathcal P$.

Now, suppose that $(H_4)$ is verified.

Then, there exist $0<p<q$, such that:
\begin{itemize}
	\item $\dfrac{f(t,u)}{u}\geqslant \dfrac{K_2^2}{K_1\,m_1\int_{a_1}^{b_1}k_1(s)\,\varPhi(s)\,ds}$ for all $t\in  I_1$ and $0< u\leqslant p$.
	\item $\dfrac{f(t,u)}{u}\leqslant \dfrac{1}{K_2\int_{a}^{b}\varPhi(s)\,ds}$ for all $t\in I$ and $u\geqslant q$.
\end{itemize}

Hence, $(H_2)$ is fulfilled for $p>0$.

Let us see that $(H_1)$ is also satisfied.

First, let us assume that $f(t,u)$ is a bounded function. That is, there exists $N>0$ such that $0\leqslant f(t,u)\leqslant N$ for all $t\in I$ and $0\leqslant u<\infty$. Then, let us choose
\[r\geqslant N\,K_2\,\int_a^b\varPhi(s)\,ds \,,\]
such that
\[f(t,u)\leqslant N\leqslant \dfrac{r}{K_2\,\int_a^b\varPhi(s)\,ds}\,,\quad \forall t\in I\,,\ 0\leqslant u\leqslant r\,.\]

Thus, $(H_1)$ holds for this $r>0$.

Now, suppose that $f$ is not bounded. Then, there exists $t_0\in I$ and $\bar r\geqslant q$ such that $f(t,u)\leqslant f(t_0,\bar r)$ for all $t\in I$ and $0\leqslant u\leqslant \bar r$. Thus, by using the hypothesis, we have
\[f(t,u)\leqslant f(t_0,\bar r)\leqslant\dfrac{\bar r}{K_2\int_{a}^{b}\varPhi(s)\,ds}\,,\quad \forall t\in I\,,\ 0\leqslant u\leqslant \bar r\,.\]

Therefore, $(H_1)$ is fulfilled for such $\bar r>0$, and the result follows again from Theorem \ref{T::2}.
\end{proof}

To finish this section, we obtain the existence of at least two fixed points of the integral operator $\mathcal L$.

Let us consider the following conditions on $f$:
\begin{itemize}
	\item[$(H_1^*)$] There exists $p>0$ such that  $(H_1)$ is fulfilled and
	\[f(t,p)< \dfrac{p}{K_2\,\int_a^b\varPhi(s)\,ds}\,,\quad \forall t\in I\,.\]
	\item[$(H_2^*)$] There exists $q>0$ such that $(H_2)$ is fulfilled and
	\[f\left( t,q\right) > \dfrac{K_2\,q}{K_1\,\int_{a_1}^{b_1}k_1(s)\,\varPhi(s)\,ds}\,,\quad \forall t\in I_1 \,.\]
\end{itemize}

\begin{remark}
	Realize that conditions $(H_1^*)$ and $(H_2^*)$ are a small restriction of $(H_1)$ and $(H_2)$, respectively. As we will see, this restriction allows us to prove that the two fixed points that we find are, in fact, different.
\end{remark}

\begin{theorem}
	\label{T::3}
		If $G$ satisfies $(Pg_1)$, then the operator $\mathcal L$ has at least two fixed points, $u_1$ and $u_2\in\mathcal P$, provided that $f_0^-(t)=f_\infty^-(t)=\infty$ for $t\in I$ and $(H_1^*)$ holds. In such a case, $0<\|u_1\|_\infty<p<\|u_2\|_\infty$ ($p$ given in $(H_1^*)$).
\end{theorem}
\begin{proof}
As in 	Corollary \ref{Cor::2.4}, we can see that $f_0^-(t)=f_\infty^-(t)=\infty$ imply that there exist positive numbers $q_1$ and $q_2$, with $0<q_1<p<q_2$ such that $(H_2)$ is fulfilled with respect to $q_1$ and $q_2$, respectively.

Since $(H_1^*)$ is  a restriction of $(H_1)$ we have that, in particular, $(H_1)$ holds for $p$.

Hence, from Theorem \ref{T::2}, we conclude that there exist two fixed points, $u_1$ and $u_2$, such that $0<q_1\leqslant \|u_1\|_\infty\leqslant p\leqslant \|u_2\|_\infty\leqslant q_2$. 

To finish the proof, we need to ensure that $u_1$ and $u_2$ are, in fact, different. To this end, let us prove that if $u\in\mathcal P\cap\partial \Omega_p$, then $u$ cannot be a fixed point of $\mathcal L$.

Let $u\in\mathcal P\cap\partial \Omega_p$, that is, $\|u\|_\infty=p$. From $(H_1^*)$ and property $(Pg_1)$, we have:
\begin{equation}\label{Ec::Des3}\begin{split}
\|\mathcal L\,u\|_\infty&=\sup_{t\in I} \int_a^b G(t,s)\,f(s,u(s))\,ds\leqslant \sup_{t\in I}\int_a^bk_2(t)\,\varPhi(s)\,f(s,u(s))\,ds\\&=K_2\,\int_a^b\varPhi(s)\,f(s,u(s))\,ds\leqslant K_2\,\int_a^b\varPhi(s)\,\dfrac{p}{K_2\,\int_a^b\varPhi(s)\,ds}\,ds\\&=p=\|u\|_\infty\,.
\end{split}\end{equation}

Since $\|u\|_\infty=p$ we have that there exists $t_0\in I$ such that $u(t_0)=p$. Now, since $f$ is a continuous function and $\|u\|_\infty=p$, from $(H_1^*)$, there exists a neighborhood of $t_0$ such that the strict inequality given in $(H_1^*)$ is fulfilled, hence the last inequality in \eqref{Ec::Des3} is strict and  $\|\mathcal L\,u\|_\infty<\|u\|_\infty$. Thus,  $u$ cannot be a fixed point. So, there exist two fixed points, $u_1$ and $u_2$, such that $0<q_1\leqslant \|u_1\|_\infty< p<\|u_2\|_\infty$ and the result is proved.
\end{proof}

In an analogous way, we can prove the following ``dual'' result:
\begin{theorem}
	\label{T::4}
	If $G$ satisfies $(Pg_1)$, then the operator $\mathcal L$ has at least two fixed points, $u_1$ and $u_2\in\mathcal P$, provided that $f_0^+(t)=f_\infty^+(t)=0$ and $(H_2^*)$ holds. In such a case, $0<\|u_1\|_\infty<q<\|u_2\|_\infty$ ($q$ given in $(H_2^*)$).
\end{theorem}
%

\section{Existence of multiple fixed points}

In this section,  by using Theorems \ref{T::1.1} and \ref{T::1.2}, we prove the existence of two or three non-trivial fixed points, respectively, of operator $\mathcal{L}$ defined in \eqref{Ec::OpInt}. We follow the steps given in \cite{AndHof}, however we impose slightly  weaker conditions on $f$ than the ones given in that reference.

\begin{theorem}
	\label{T::5}
	Suppose that there exist positive numbers $p$, $q$ and $r$ such that $0<p<q<r$, and assume that function $f$ satisfies the following conditions:
	\begin{enumerate}[$(i)$]
		\item $f(t,u)\geqslant \dfrac{u}{m_1\int_{a_1}^{b_1}\varPhi(s)\,ds}$ for all $t\in I_1$ and $u\in\left[ r,\dfrac{K_2}{m_1}\,r\right] $, being the inequality strict at $u=r$,
		\item $f(t,u)\leqslant\dfrac{q}{K_2\int_{a}^{b}\varPhi(s)\,ds}$ for all $t\in I$ and $u\in\left[ 0,\dfrac{K_2}{m_1}\,q\right] $, being the inequality strict at $u=q$,
		\item $f(t,u)>\dfrac{K_2\,u}{K_1\int_{a_1}^{b_1}k_1(s)\,\varPhi(s)\,ds}$ for all $t\in I_1$ and $u\in\left[ \dfrac{m_1}{K_2}\,p,p\right] $.
	\end{enumerate}

Then, if $G$ satisfies $(Pg_1)$, the operator $\mathcal L$ has at least two fixed points, $u_1$ and $u_2$, such that \[p<\|u_1\|_\infty\,,\quad \max_{t\in I_1}u_1(t)<q<\max_{t\in I_1}u_2(t)\,,\quad \min_{t\in I_1}u_2(t)<r\,.\]
\end{theorem}


\begin{proof}
	The proof is based on Theorem \ref{T::1.1}.
	
	Consider:
	\begin{align}
	\label{Ec::alpha}\alpha(u)&:=\min_{t\in I_1}u(t)\,,\\
	\label{Ec::theta} \theta(u)&:=\max_{t\in I_1}u(t)\,,\\
\nonumber	\text{and\qquad}&\hspace{11cm}\\
	\label{Ec::gamma}\gamma(u)&:=\|u\|_\infty\,.
	\end{align}
	
	Clearly, $\alpha(u)\leqslant \theta(u)\leqslant \gamma(u)$ for all $u\in\mathcal P$.
	
	Since $u\in \mathcal P$, then
	\[\alpha(u)=\min_{t\in I_1}u(t)\geqslant \min_{t\in I_1} \dfrac{k_1(t)}{K_2}\,\|u\|_\infty=\dfrac{m_1}{K_2}\,\gamma(u)\,,\]
	that is, $\gamma(u)\leqslant \dfrac{K_2}{m_1}\,\alpha(u)$ for all $u\in\mathcal P$.
	
	Moreover, for all $\lambda\in \mathbb{R}$ and $u\in\mathcal P$, we have:	
	\[\theta(\lambda\,u)=\max_{t\in I_1}\left\lbrace \lambda\,u(t)\right\rbrace =\lambda\,\max_{t\in I_1}u(t)=\lambda\,\theta(u)\,.\]
	
	As we have noticed, $\mathcal L\colon \mathcal P\longrightarrow\mathcal P$ is a completely continuous operator.
	
	Moreover, let us consider $u\in \partial \mathcal P(\alpha,r)$, i.e. $\min_{t\in I_1}u(t)=r$. we have that \[\alpha(u)=r\geqslant \dfrac{m_1}{K_2}\|u\|_\infty\,.\]
	
	Thus, from $(i)$, the following inequalities are verified:
{\small 	\begin{equation}\label{Ec::Des4}\begin{split}
	\alpha(\mathcal L\,u)&=\min_{t\in I_1}\int_a^bG(t,s)\,f(s,u(s))\,ds\geqslant \min_{t\in I_1}\int_a^bk_1(t)\,\varPhi(s)\,f(s,u(s))\,ds\\&\geqslant \min_{t\in I_1} k_1(t)\int_{a_1}^{b_1}\varPhi(s)\,f(s,u(s))\,ds\geqslant m_1\int_{a_1}^{b_1}\varPhi(s)\,\dfrac{u(s)}{m_1\,\int_{a_1}^{b_1}\varPhi(s)\,ds}\,ds\,.
	\end{split}\end{equation}}	

On the other hand	$\alpha(u)=r$ implies that there exists $t_1\in I_1$ such that $u(t_1)=r$. Since $f$ is a continuous function on $I_1\subset I$, from $(i)$, there exist $I_0\subset I_1$ a nontrivial subinterval where the inequality given on $(i)$  is strict. Thus, the last inequality of \eqref{Ec::Des4} is also strict and we have for all $u\in \partial \mathcal{P}(\alpha,r)$:
	\[\alpha(\mathcal L\,u)> m_1\int_{a_1}^{b_1}\varPhi(s)\,\dfrac{u(s)}{m_1\,\int_{a_1}^{b_1}\varPhi(s)\,ds}\,ds\geqslant r\,.\]

	Now, for $u\in \partial \mathcal P(\theta, q)$, i.e. $\max_{t\in I_1}u(t)=q$, we have that \[\gamma(u)\geqslant\theta (u)=q\geq\alpha(u)\geqslant \dfrac{m_1}{K_2}\gamma(u)\,.\]
	
	 Thus, $q\leqslant \|u\|_\infty\leqslant \dfrac{K_2}{m_1}q$ and from $(ii)$, we have:
	\begin{equation}\label{Ec::Des5}\begin{split}
	\theta (\mathcal L\,u)&=\max_{t\in I_1}\int_a^b G(t,s)\,f(s,u(s))\,ds\leqslant \max_{t\in I_1}\int_a^bk_2(t)\,\varPhi(s)\,f(s,u(s))\,ds\\&\leqslant K_2\,\int_a^b\varPhi(s)\,f(s,u(s))\,ds\leqslant K_2\int_a^b\varPhi(s)\dfrac{q}{K_2\int_a^b\varPhi(s)\,ds}\,ds\,.
	\end{split}\end{equation}
	
		Repeating the previous arguments, $\theta(u)=q$ implies that there exists $t_1\in I_1\subset I$ such that $u(t_1)=q$. Since $f$ is a continuous function on $I$, from $(ii)$, there exist $I_0\subset I_1$ a nontrivial subinterval where the inequality given on $(ii)$  is strict. Thus, the last inequality of \eqref{Ec::Des5} is also strict and we have for all $u\in\partial \mathcal P(\theta,q)$:
	\[\theta(\mathcal L\,u)< K_2\int_a^b\varPhi(s)\dfrac{q}{K_2\int_a^b\varPhi(s)\,ds}\,ds=q\,.\]

	Finally, $\mathcal P(\gamma,p)=\{u\in\mathcal P\ \mid \|u\|_\infty<p\}\neq \emptyset$ and for all $u\in \partial\mathcal P(\gamma,p)$, we have that $\|u\|_\infty=p$ and $\alpha(u)\geqslant \frac{m_1}{K_2}\,p$. 
	
	Thus, from $(iii)$, we obtain:
		\[\begin{split}
	\gamma(\mathcal L\,u)&=\sup_{t\in I}\int_a^bG(t,s)\,f(s,u(s))\,ds\geqslant \sup_{t\in I}\int_a^bk_1(t)\,\varPhi(s)\,f(s,u(s))\,ds\\&\geqslant K_1\int_{a_1}^{b_1}\varPhi(s)\,f(s,u(s))\,ds\geqslant K_1\int_{a_1}^{b_1}\,\varPhi(s)\,\dfrac{K_2\,u(s)}{K_1\,\int_{a_1}^{b_1}k_1(s)\varphi(s)\,ds}\\&\geqslant \int_{a_1}^{b_1}\varPhi(s)\,\dfrac{K_2\,\frac{k_1(s)}{K_2}\,\|u\|_\infty}{\int_{a_1}^{b_1}k_1(s)\,\varPhi(s)\,ds}\,ds=p\,.
	\end{split}\]
	
	So, we conclude that $\gamma(\mathcal L\,u)>p$ for $u\in \partial \mathcal P(\gamma,p)$.
	
	Hence, all the hypotheses of Theorem \ref{T::1.1} are fulfilled. Thus, $\mathcal L$ has at least two fixed points on $\mathcal P$, $u_1$ and $u_2$, such that $p<\gamma(u_1)=\|u\|_\infty$ and $q>\theta(u_1)=\max_{t\in I_1}u_1(t)$. Moreover, $q<\theta(u_2)=\max_{t\in I_1}u_2(t)$ and $r>\alpha(u_2)=\min_{t\in I_1} u_2(t)$ and the result is proved.	
\end{proof}

\begin{remark}
	Realize that in the third item of Theorem \ref{T::5}, we cannot avoid the strict inequality on the whole interval, since $\|u\|_\infty=p$ implies that there exists $t_1\in I$ such that $u(t_1)=p$, however we cannot ensure that $p\in I_1$.
	
	Even though, in most of cases it is fulfilled that the point $t_1\in I$ for which $u(t_1)=p$ belongs to $I_1$. This is due because $I_1$ is built to avoid having zeros or small values of $u$. It is convenient to have $I_1$ as larger as possible to ensure that the obtained bounds are better, so $t_1$ is likely to be in $I_1$. In such a case, we can ask for the strict inequality only at $u=p$.
\end{remark}

Now, as an application of Theorem \ref{T::1.2}, we obtain the next result that ensures  the existence of at least three critical points of operator $\mathcal L$.

\begin{theorem}
	\label{T::6}
	Let $p$, $q$ and $r$ be positive numbers satisfying the relation: 
	\[0<p<q<\dfrac{K_2}{m_1}q\leqslant r\,.\]

Assume, moreover, that the function $f$ satisfies the following conditions:
	\begin{enumerate}[$(a)$]
			\item $f(t,u)\leqslant\dfrac{r}{K_2\int_{a}^{b}\varPhi(s)\,ds}$ for all $t\in I$ and $u\in\left[ 0,r\right] $,
		\item $f(t,u)<\dfrac{p}{K_2\int_{a}^{b}\varPhi(s)\,ds}$ for all $t\in I$ and $u\in\left[ 0,p\right] $,
		\item $f(t,u)\geqslant\dfrac{u}{m_1\int_{a_1}^{b_1}\varPhi(s)\,ds}$ for all $t\in I_1$ and $u\in\left[ q,\dfrac{K_2}{m_1}q\right] $, being the inequality strict for $u=q$.
	\end{enumerate}

Then, if $G$ satisfies $(Pg_1)$, the operator $\mathcal L$ has at least three fixed points $u_1$, $u_2$, $u_3\in \{u\in\mathcal P\ \mid\ \|u\|_\infty\leqslant r\}$ such that  $\max_{t\in I_1}u_1(t)<p$, $q<\min_{t\in I_1}u_2(t)$ and $p<\max_{t\in I_1}u_3(t)$ with $\min_{t\in I_1}u_3(t)<q$.
\end{theorem}

\begin{proof}
The  proof follows from Theorem \ref{T::1.2}. 

We consider $\alpha$, $\theta$ and $\gamma$ as in \eqref{Ec::alpha}--\eqref{Ec::gamma} and, moreover $\Psi(u)=\alpha(u)$ and $\beta(u)=\theta(u)$. Clearly $\alpha$ and $\Psi$ are concave and non-negative functionals in $\mathcal P$. And $\beta$, $\theta$ and $\gamma$ are convex	and non-negative functionals in $\mathcal P$.

It is proved in Theorem \ref{T::2} that $\mathcal L(\mathcal P)\subset \mathcal P$. 

Let us see now that $\mathcal L\left( \overline{\mathcal P(\gamma,r)}\right) \subset \overline{\mathcal P(\gamma,r)}$. Indeed, let $u\in \overline{\mathcal P(\gamma,r)}$ (i.e. $\|u\|_\infty\leqslant r$), from $(a)$ we have: 
\[\begin{split}
\left\| \mathcal L\,u\right\| _\infty&=\sup_{t\in I}\int_a^bG(t,s)\,f(s,u(s))\,ds\leqslant \sup_{t\in I}\int_a^bk_2(t)\,\varPhi(s)\,f(s,u(s))\,ds\\&=K_2\int_a^b\varPhi(s)\,f(s,u(s))\,ds\leqslant K_2\int_a^b\varPhi(s)\,\dfrac{r}{K_2\int_a^b\varPhi(s)\,ds}\,ds=r\,,
\end{split}\]
thus, $\mathcal L\,u\in \overline{\mathcal P(\gamma,r)}$ and we conclude that $\mathcal L\left( \overline{\mathcal P(\gamma,r)}\right) \subset \overline{\mathcal P(\gamma,r)}$.

Obviously, $\alpha(u)\leqslant \beta(u)$ and $\gamma(u)=\|u\|_\infty$.

We consider $u_q(t)=\dfrac{K_2}{m_1}\,q$, it is obvious that $u_q$ belongs to the following set $\left\lbrace u\in \mathcal P\left( \gamma,\theta,\alpha,q,\dfrac{K_2}{m_1}\,q,r\right) \ \mid \alpha(u)>q\right\rbrace$, described as follows:
\[\left\lbrace u\in\mathcal P\ \mid q<\min_{t\in I_1}u(t)\,,\ \max_{t\in I_1}u(t)\leqslant \dfrac{K_2}{m_1}q\,,\ \|u\|_\infty\leqslant r\right\rbrace\neq \emptyset\,. \]

Let $u\in \mathcal P\left( \gamma,\theta,\alpha,q,\dfrac{K_2}{m_1}\,q,r\right) $, then using $(c)$, we obtain:
\begin{equation}\label{Ec::alfa}\begin{split}
\alpha(\mathcal L\,u)&=\min_{t\in I_1}\int_a^bG(t,s)\,f(s,u(s))\,ds\geqslant \min_{t\in I_1}\int_a^bk_1(t)\,\varPhi(s)\,f(s,u(s))\,ds\\&\geqslant m_1\int_{a_1}^{b_1}\varPhi(s)\,f(s,u(s))\,ds\geqslant m_1\int_{a_1}^{b_1}\varPhi(s)\,\dfrac{u(s)}{m_1\int_{a_1}^{b_1}\varPhi(s)\,ds}\,ds\,.
\end{split}\end{equation}

If there exists $s_1\in I_1$ such that $u(s_1)>q$, then it must exist a nontrivial subinterval of $I_1$ where $u(s)>q$ for all $s$ in such a subinterval. Then, directly from \eqref{Ec::alfa}, we have $\alpha (\mathcal L\,u)>q$.

Now, if $u(s)=q$ for all $s\in I_1$, from $(c)$ and \eqref{Ec::alfa}, we obtain:
\begin{equation*}
\alpha(\mathcal L\,u)\geqslant m_1\int_{a_1}^{b_1}\varPhi(s)\,f(s,q)\,ds> m_1\int_{a_1}^{b_1}\varPhi(s)\,\dfrac{q}{m_1\int_{a_1}^{b_1}\varPhi(s)\,ds}\,ds=q\,.
\end{equation*}

So, the hypothesis $a)$ on Theorem \ref{T::1.2} is fulfilled. Now, let us see $b)$.

If we consider the function $u_p(t)=\dfrac{m_1}{K_2}\,p$, it is clear that it belongs to the  set $\left\lbrace u\in \mathcal P\left( \gamma,\beta,\Psi,\dfrac{m_1}{K_2}\,p,p,r\right) \ \mid \beta(u)<p\right\rbrace$, which is described as follows:
\[\left\lbrace u\in\mathcal P\ \mid \dfrac{m_1}{K_2}\,p\leqslant\min_{t\in I_1}u(t)\,,\ \max_{t\in I_1}u(t)< p\,,\ \|u\|_\infty\leqslant r\right\rbrace\neq \emptyset\,. \]

Let $u\in \mathcal P\left( \gamma,\beta,\Psi,\dfrac{m_1}{K_2}\,p,p,r\right) $, from $(b)$, we have

\[\begin{split}
\beta(\mathcal L\,u)&=\max_{t\in I_1}\int_a^bG(t,s)\,f(s,u(s))\,ds\leqslant \max_{t\in I_1}\int_a^bK_2\,\varPhi(s)\,f(s,u(s))\,ds\\&<K_2\int_a^b\varPhi(s)\,\dfrac{p}{K_2\int_a^b\varPhi(s)\,ds}\,ds=p\,.
\end{split}\]

Thus, $\beta(\mathcal L\,u)<p$ for all $u\in \mathcal P\left( \gamma,\beta,\Psi,\dfrac{m_1}{K_2}\,p,p,r\right) $ and, as consequence, condition $b)$ in Theorem \ref{T::1.2} is satisfied.

Now, let $u\in \mathcal P(\gamma,\alpha,q,r)$ be such that $\theta (\mathcal L\,u)>q$. Then
\[\begin{split}
\alpha(\mathcal L\,u)&=\min_{t\in I_1} \int_a^b G(t,s)\,f(s,u(s))\,ds\geqslant \min_{t\in I_1}\int_a^b k_1(t)\,\varPhi(s)\,f(s,u(s))\,ds\\&=\dfrac{m_1}{K_2}\int_a^bK_2\,G(t,s)\,f(s,u(s))\,ds\geqslant \dfrac{m_1}{K_2}\int_a^b\max_{t\in I_1}G(t,s)\,f(s,u(s)\,ds\\&\geqslant\dfrac{m_1}{K_2}\max_{t\in I_1}\int_a^bG(t,s)\,f(s,u(s))\,ds=\dfrac{m_1}{K_2}\theta(\mathcal L\,u)>\dfrac{m_1}{K_2}\,q\,.
\end{split}\]

Finally, take $u\in \mathcal Q(\gamma,\beta,p,r)$ such that $\Psi(\mathcal{L}\,u)<\dfrac{m_1}{K_2}p$, then
\[\begin{split}
\beta(\mathcal L\,u)&=\max_{t\in I_1} \int_a^b G(t,s)\,f(s,u(s))\,ds\leqslant \int_a^b K_2\,\varPhi(s)\,f(s,u(s))\,ds\\&=\dfrac{K_2}{m_1}\int_a^bm_1\,\varPhi(s)\,f(s,u(s))\,ds\leqslant \dfrac{K_2}{m_1}\int_a^b\min_{t\in I_1}G(t,s)\,f(s,u(s)\,ds\\&\leqslant\dfrac{K_2}{m_1}\min_{t\in I_1}\int_a^bG(t,s)\,f(s,u(s))\,ds=\dfrac{K_2}{m_1}\Psi(\mathcal L\,u)<\dfrac{K_2}{m_1}\dfrac{m_1}{K_2}\,p=p\,.
\end{split}\]

Therefore, all the hypotheses of Theorem \ref{T::1.2} are fulfilled and we have ensured the existence of at least three critical points such that $p>\beta(u_1)=\max_{t\in I_1}u_1(t)$, $q<\alpha(u_2)=\min_{t\in I_1}u_2(t)$ and $p<\beta(u_3)=\max_{t\in I_1}u_3(t)$ with $q>\alpha(u_3)=\min_{t\in I_1}u_3(t)$. And the proof is complete.
\end{proof}

\section{Existence results for $(k,n-k)$ problems}

Let us consider the $n^{\mathrm{th}}-$order differential operator introduced in \eqref{Ec::Tn}.

As we have mentioned before, the existence of solution of the problem \eqref{Ec::Dif}-\eqref{Ec::BC} is equivalent to the existence of a fixed point of the integral operator \eqref{Ec::OpIn}. 

Let us denote
\[X_k=\{u\in C^n(I)\ \mid\ u(a)=\cdots=U^{(k-1)}(a)=u(b)=\cdots=u^{(n-k-1)}(b)=0\}\,,\]
the set of functions related to the boundary conditions \eqref{Ec::BC}.

Let us introduce a concept about the maximum number of zeros of the solutions of a linear differential equation.
\begin{definition}\cite[Chapter 0, Definition 1]{Cop}
	Let $p_k\in C^{n-k}(I)$ for $k=1,\dots,n$. The $n^{\mathrm{th}}$-order linear differential equation $T_n[\bar{M}]\,u(t)=0$  is said to be disconjugate on an interval $I$ if every non trivial solution has, at most, $n-1$ zeros on $I$, multiple zeros being counted according to their multiplicity.
\end{definition}

In \cite{CabSaa}, the main properties of te related Green's function, $g_M(t,s)$, are studied.  Now, we prove:
	\begin{theorem}\label{T::7}
	Let $\bar{M}\in\mathbb{R}$ be such that equation $T_n[\bar{M}]\,u(t)=0$  is disconjugate on $I$. Then the  following properties are fulfilled:\\
	
	If $n-k$ is even and $2\leqslant k \le n-1$, then $g_M(t,s)$ verifies $(Pg_1)$ if, and only if, $M\in (\bar{M}-\lambda_1,\bar{M}-\lambda_2]$, where:
	\begin{itemize}
		\item $\lambda_1>0$ is the least positive eigenvalue  of operator $T_n[\bar{M}]$ in $X_k$.
		\item $\lambda_2<0$ is the maximum of:
		\begin{itemize}
			\item $\lambda_2'<0$, the biggest negative eigenvalue  of operator $T_n[\bar{M}]$ in $X_{k-1}$.
			\item $\lambda_2''<0$, the biggest negative eigenvalue of operator $T_n[\bar{M}]$ in $X_{k+1}$.\\
		\end{itemize}
	\end{itemize}
	
	If $k=1$ and $n$ is odd, then  $g_M(t,s)$ verifies $(Pg_1)$ if, and only if, $M\in (\bar{M}-\lambda_1,\bar{M}-\lambda_2]$, where:
	\begin{itemize}
		\item $\lambda_1>0$ is the least positive eigenvalue  of operator $T_n[\bar{M}]$ in $X_1$.
		\item $\lambda_2<0$ is the biggest negative eigenvalue of operator $T_n[\bar{M}]$ in $X_{2}$.\\
		
	\end{itemize}
	
	If $n-k$ is odd and $2\leqslant k\leqslant n-2$, then  $-g_M(t,s)$ verifies $(Pg_1)$ if, and only if, $M\in [\bar{M}-\lambda_2,\bar{M}-\lambda_1)$, where:
	\begin{itemize}
		\item $\lambda_1<0$ is the biggest negative eigenvalue of operator $T_n[\bar{M}]$ in $X_k$.
		\item $\lambda_2>0$ is the minimum of:
		\begin{itemize}
			\item $\lambda_2'>0$, the least positive eigenvalue  of operator $T_n[\bar{M}]$ in $X_{k-1}$.
			\item $\lambda_2''>0$, the least positive eigenvalue  of operator $T_n[\bar{M}]$ in $X_{k+1}$.\\
		\end{itemize}
	\end{itemize}
	
	If $k=1$ and $n>2$ is even , then  $-g_M(t,s)$ verifies $(Pg_1)$ if, and only if, $M\in [\bar{M}-\lambda_2,\bar{M}-\lambda_1)$, where:
	\begin{itemize}
		\item $\lambda_1<0$ is the biggest negative eigenvalue of operator $T_n[\bar{M}]$ in $X_1$.
		\item $\lambda_2>0$ is  the least positive eigenvalue  of operator $T_n[\bar{M}]$ in $X_{2}$.\\
	\end{itemize}
	
	If $k=n-1$ and $n>2$, then  $-g_M(t,s)$ verifies $(Pg_1)$ if, and only if,  $M\in [\bar{M}-\lambda_2,\bar{M}-\lambda_1)$, where:
	\begin{itemize}
		\item $\lambda_1<0$ is the biggest negative eigenvalue of operator $T_n[\bar{M}]$ in $X_{n-1}$.
		\item $\lambda_2>0$ is  the least positive eigenvalue  of operator $T_n[\bar{M}]$ in $X_{n-2}$.\\
	\end{itemize}
	
	If $n=2$, then  $-g_M(t,s)$ verifies $(Pg_1)$ if, and only if, $M\in (-\infty,\bar{M}-\lambda_1)$, where:
	\begin{itemize}
		\item $\lambda_1<0$ is the biggest negative eigenvalue of operator $T_2[\bar{M}]$ in $X_1$.\\
	\end{itemize}
	\end{theorem}
\begin{proof}
	From the proof of \cite[Theorem 3.1]{CabSaa} we obtain that  the following inequalities are fulfilled for those $M$  which the result refers to.
	\[\begin{split}
	(\infty>)(-1)^{n-k}g_M(t,s)&>0\,,\quad \forall\, (t,s)\in (a,b)\times(a,b)\,,\\
	(\infty>)(-1)^{n-k}\dfrac{\partial^kg_M(t,s)}{\partial t^k}_{\mid t=a}&>0\,,\quad \forall\, s\in(a,b)\,,\\
	(\infty>)(-1)^{n-k}\dfrac{\partial^{n-k}g_M(t,s)}{\partial s^{n-k}}_{\mid s=a}&>0\,,\quad \forall\, t\in (a,b)\,,\\
	(\infty>)\dfrac{\partial^{n-k}g_M(t,s)}{\partial t^{n-k}}_{\mid t=b}&>0\,,\quad \forall\, s\in (a,b)\,,\\
	(\infty>)(-1)^n\dfrac{\partial^k g_M(t,s)}{\partial s^k}_{\mid s=b}&>0\,,\quad \forall\, t\in (a,b)\,.
	\end{split}\]
	
	Now, we construct the following function:
	\[\varPhi(s)=(s-a)^{n-k}\,(b-s)^k>0\,,\forall\, s\in(a,b)\,,\]
	and we have that
	\[u(t,s)=(-1)^{n-k}\dfrac{g_M(t,s)}{\varPhi(s)}>0\,,\forall\,(t,s)\in(a,b)\times(a,b)\,.\]
	
	Moreover, the following limits are real and:
	\[\begin{split}
	\lim_{s\rightarrow a^+}(-1)^{n-k}\dfrac{g_M(t,s)}{\varPhi(s)}&=\dfrac{(-1)^{n-k}\dfrac{\partial^{n-k}g_M(t,s)}{\partial s^{n-k}}_{\mid s=a}}{(n-k)!\,(b-a)^{k}}>0\,,\quad \forall\,t\in(a,b)\,,\\
	\lim_{s\rightarrow b^-}(-1)^{n-k}\dfrac{g_M(t,s)}{\varPhi(s)}&=\dfrac{(-1)^{n-k}\dfrac{\partial^{k}g_M(t,s)}{\partial s^{k}}_{\mid s=b}}{k!\,(-1)^k(b-a)^{n-k}}\\&=\dfrac{(-1)^{n}\dfrac{\partial^{k}g_M(t,s)}{\partial s^{k}}_{\mid s=b}}{k!(b-a)^{n-k}}>0\,,\quad \forall\,t\in(a,b)\,.
	\end{split}\]
	
	Hence, we can define $\tilde u(t,s)$, as the continuous extension of $u(t,s)$ to $(a,b)\times I$, and we have
	\[\begin{split}
	k_1(t)&:=\min_{s\in I}\tilde u(t,s)>0\,,\quad\forall\,t\in(a,b)\,,\\
	k_2(t)&:=\max_{s\in I}\tilde u(t,s)>0\,,\quad \forall\,t\in(a,b)\,.
	\end{split}\]
	Clearly, $0<k_1(t)<k_2(t)$ for all $t\in(a,b)$ and $(-1)^{n-k}\,g_M(t,s)$ fulfills condition $(P_{g1})$. Then, the result is proved.
\end{proof}
\subsection{Particular cases}
This section is devoted to apply previous results for some particular cases with fixed $(k,n-k)$ boundary conditions. We distinguish the cases depending on the value of $n$.
\begin{itemize}
	\item Second order: $n=2$ and $k=1$.
\end{itemize}
Let us study a second order operator with constant coefficients
\begin{equation}
\label{Ec::2or} T_2[B,M]\,u(t)=u''(t)+B\,u'(t)+M\,u(t)\,,\quad t\in  [0,1]\,,
\end{equation}
for a fixed $B\in\mathbb{R}$.

We consider the space of definition related to the boundary conditions $(1,1)$. That is,
\[X_1=\left\lbrace u\in C^2([0,1])\ \mid\ u(0)=u(1)=0\right\rbrace .\]

As it is proved in  \cite{CabSaa}, the biggest negative eigenvalue of $T_2[B,0]$ in $X_1$ is $\lambda_B=-\dfrac{B^2+4\,\pi^4}{4}$. Moreover, $T_2[B,0]\,u(t)=0$ is a disconjugate equation on $[0,1]$ for all $B\in\mathbb{R}$.

Thus, from Theorem \ref{T::7}, $-g_M(t,s)$, satisfies the property $(Pg_1)$ if, and only if, $M\in\left( -\infty, \dfrac{B^2+4\,\pi^2}{4}\right)$. In particular such a property is fulfilled for $M=0$. 

Hence, let us study the following problem:
\begin{equation}\label{Ec::2ord}
\begin{split}
u''(t)+B\,u'(t)+f(t,u(t))=0\,,\ t\in[0,1]\,,\quad
u(0)=u(1)=0\,.
\end{split}
\end{equation}

Clearly, $\varPhi(s)=s(1-s)$.

\vspace{0.5cm}

First, consider $B=0$. We have, (see \cite{Cab,CaCiMa})
\[-g(t,s)=\begin{cases}
s\,(1-t)\,,&0\leqslant s\leqslant t\leqslant1\\\\
(1-s)\,t\,,&0<t<s\leqslant1\,.
\end{cases}\]
%

By direct calculations, we obtain $k_2(t)=1=K_2$ and
\[k_1(t)=\min_{t\in[0,1]}\{1-t,t\}=\begin{cases}
t\,,&0\leqslant t\leqslant \dfrac{1}{2}\,,\\\\
1-t\,,&\dfrac{1}{2}<t\leqslant 1\,,
\end{cases}\]
and $K_1=k_1\left( \frac{1}{2}\right) =\dfrac{1}{2}$.

In Figure \ref{Fig::1}, it is represented the function $\tilde{u}(t,s)=-\dfrac{g(t,s)}{\varPhi(s)}$, bounded from above by $k_2(t)$ and from below by $k_1(t)$.  Moreover, in Figure \ref{Fig::2}, it is given the same representation considering the constant values $t_0=\frac{2}{3}$ and $s_0=\frac{2}{3}$, respectively.
	\begin{figure}[h]
	\centering
	\includegraphics[width=0.7\textwidth]{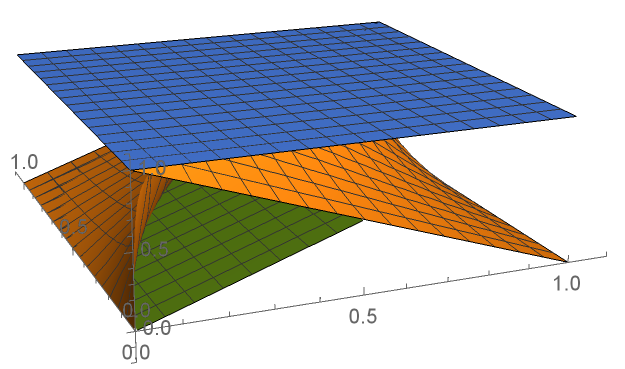}
	\caption{\scriptsize$\tilde u(t,s)$ (orange) bounded from above by $k_2(t)$ (blue) and from below by $k_1(t)$ (green).\label{Fig::1}}
\end{figure}

	\begin{figure}[h]
	\centering
	\includegraphics[width=0.4\textwidth]{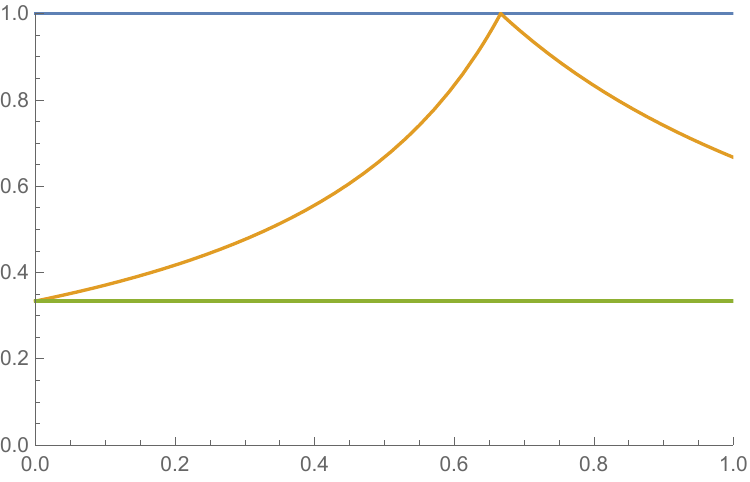} \hspace{2cm}\includegraphics[width=0.4\textwidth]{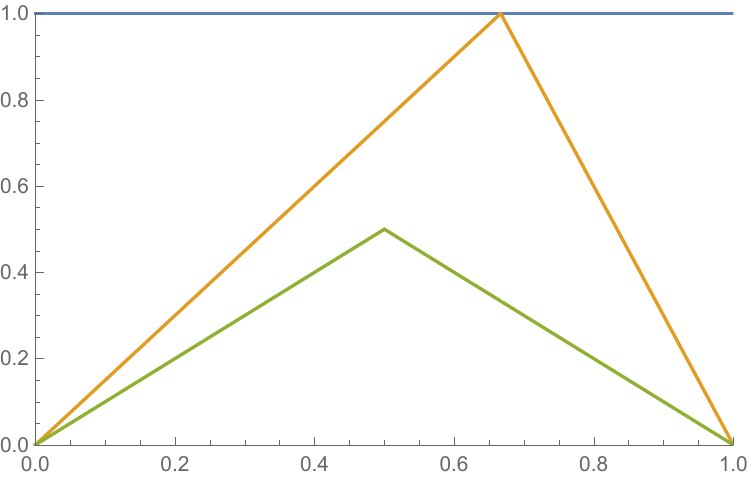} 
	\caption{\scriptsize Figure \ref{Fig::1} for $t_0=\frac{2}{3}$ on the left and for $s_0=\frac{2}{3}$ on the right.\label{Fig::2}}
\end{figure}
Let us consider $I_1=[a_1,b_1]=\left[ \dfrac{1}{4},\dfrac{3}{4}\right]$, then $m_1=\dfrac{1}{4}$, and it is immediate to verify that
\[\begin{split}
\int_0^1\varPhi(s)\,ds&=\dfrac{1}{6}\,,\\
\int_{1/4}^{3/4}\varPhi(s)\,ds&=\dfrac{11}{96}\,,\\
\int_{1/4}^{3/4}k_1(s)\,\varPhi(s)\,ds&=\dfrac{67}{1536}\,.
\end{split}\]

Thus, we are able to ensure the existence of at least two or three solutions of problem \eqref{Ec::2ord} by means of the previously obtained fixed point theorems. First, let us write the corresponding assumptions $(H_1)$ and $(H_2)$ for this situation,
\begin{itemize}
	\item[$(H_1)$] There exists $p>0$ such that $f(t,u)\leqslant 6\,p$ for all $t\in [0,1]$ and $u\in[0,p]$.
	
	\item[$(H_2)$] There exists $q>0$ such that $f(t,u)\geqslant \dfrac{3072}{67}u$ for all $t\in\left[ \dfrac{1}{4},\dfrac{3}{4}\right] $ and $u\in\left[ \dfrac{q}{4},q\right] $.
\end{itemize}

Finally,  as a direct consequence of Theorems \ref{T::5} and \ref{T::6}, we obtain the following results:

\begin{theorem}\label{T::5.3}
	Let $I=[0,1]$ and $I_1=\left[ \frac{1}{4},\frac{3}{4}\right] $ and suppose that there exist positive numbers $p$, $q$ and $r$ such that $0<p<q<r$, and assume that function $f$ satisfies the following conditions:
	\begin{enumerate}[$(i)$]
		\item $f(t,u)\geqslant \dfrac{384}{11}u$ for all $t\in I_1 $ and $u\in\left[ r,4\,r\right] $, being the inequality strict at $u=r$,
		\item $f(t,u)\leqslant 6\,q$ for all $t\in I$ and $u\in\left[ 0,4\,q\right] $, being the inequality strict at $u=q$,
		\item $f(t,u)>\dfrac{3072}{67}u$ for all $t\in I_1 $ and $u\in\left[ \dfrac{p}{4},p\right] $.
	\end{enumerate}
	
	Then, for $B=0$, problem \eqref{Ec::2ord} has at least two positive solutions,  $u_1$ and $u_2$, such that \[p<\|u_1\|_\infty\,,\ \max_{t\in I_1}u_1(t)<q\text{ and } q<\max_{t\in I_1}u_2(t)\,,\ \min_{t\in I_1}u_2(t)<r\,.\]
\end{theorem}

\begin{theorem}\label{T::5.4}
	Let $I=[0,1]$ and $I_1=\left[ \frac{1}{4},\frac{3}{4}\right] $ and suppose $p$, $q$ and $r$ are positive numbers such that 
	\[0<p<q<4\,q\leqslant r\,,\]
	such the function $f$ satisfies the following conditions:
	\begin{enumerate}[$(a)$]
		\item $f(t,u)\leqslant 6\,r$ for all $t\in I$ and $u\in\left[ 0,r\right] $,
		\item $f(t,u)<6\,p$ for all $s\in I$ and $u\in\left[ 0,p\right] $,
		\item $f(t,u)\geqslant\dfrac{3072}{67}u$ for all $s\in I_1 $ and $u\in\left[ q,4\,q\right] $, being the inequality strict at $u=q$.
	\end{enumerate}
	
	Then, for $B=0$, problem \eqref{Ec::2ord} has at least three positive solutions,  $u_1$, $u_2$, $u_3\in \{u\in\mathcal P\ \mid\ \|u\|_\infty\leqslant r\}$, such that  $\max_{t\in I_1}u_1(t)<p$, $q<\min_{t\in I_1}u_2(t)$ and $p<\max_{t\in I_1}u_3(t)$ with $\min_{t\in I_1}u_3(t)<q$.
\end{theorem}
\vspace{0.5cm}

Now, let us see what happens for $B>0$. In this case, the expression of the Green's function (see \cite{CaCiMa}) is given by
\[-g(t,s)=\begin{cases}
\dfrac{\left( e^{B\,s}-1\right) \left( e^{B\,(1-t)}-1\right) }{B\,(e^B-1)}\,,&0\leqslant s\leqslant t\leqslant 1\,,\\\\
\dfrac{\left( e^{B}-e^{B\,s}\right) \left( 1-e^{-B\,t}\right) }{B\,(e^B-1)}\,,&0<t<s\leqslant 1\,.
\end{cases}\]

After routine calculations we obtain, 
\[k_2(t)=\dfrac{(e^{B}-e^{B\,t})\,(1-e^{-B\,t})}{t\,(1-t)\,B\,(e^B-1)}\leqslant k_2(1)=1=K_2\,.\]

 Moreover, we have
\[k_1(t)=\begin{cases}
\dfrac{e^B-e^{B\,(1-t)}}{e^B-1}\,,&0\leqslant t\leqslant t_1=1-\dfrac{\log\left( \frac{1+e^B}{2}\right) }{B}\,,\\\\
\dfrac{e^{B\,(1-t)}-1}{e^B-1}\,,&t_1<t\leqslant 1\,.
\end{cases}\]

Thus, $K_1=\max_{t\in[0,1]}k_1(t)=k_1(t_1)=\dfrac{1}{2}$.

If we choose \begin{equation}
\label{Ec::I1}
I_1=[a_1,b_1]=\left[ 1-\dfrac{\log\left( \frac{1+3\,e^B}{4}\right) }{B},1-\dfrac{\log\left( \frac{3+e^B}{4}\right) }{B}\right]\,,\end{equation} we obtain $m_1=k_1(a_1)=k_1(b_1)=\dfrac{1}{4}$.

\begin{remark}\label{R::I11}
	For $B\in[0,2]$, both $a_1$ and $b_1$ are decreasing functions of $B$. Moreover,
	\[	a_1\in\left[ 1-\frac{1}{2} \log \left(\frac{1}{4} \left(1+3 e^2\right)\right),\frac{1}{4}\right] \subset\left[ \frac{3}{25},\frac{1}{4}\right] \,,\]
	and
	\[	b_1\in\left[ 1-\frac{1}{2} \log \left(\frac{1}{4} \left(3+e^2\right)\right),\frac{3}{4}\right] \subset\left[ \frac{13}{25},\frac{3}{4}\right] \,.\]
\end{remark}
To attain the rest of the constants involved on the different results, we use approximations of the values of the integrals. To this end we need to fix the value of $B\in[0,2]$. 

The bounds shown above can be obtained for all $B\in[0,2]$, in order to simplify the description of $k_1$, let us choose $B$ such that $t_1=\dfrac{1}{3}$, that is,  $B=\log(2+\sqrt{5})\in[0,2]$. In such a case, we have \[a_1=\dfrac{\log(\sqrt{5}-1)}{\log{(2+\sqrt{5})}}\ \text{ and } b_1=\dfrac{\log\left( 1+\frac{3}{\sqrt{5}}\right) }{\log(2+\sqrt{5})}\,.\]

 Moreover, in this case, we obtain
\[\begin{split}
\int_{a_1}^{b_1}k_1(s)\,\varPhi(s)\,ds&\approxeq 0.035872>\dfrac{3587}{100000}\,,\\
\int_{a_1}^{b_1}\varPhi(s)\,ds&\approxeq 0.095719>\dfrac{957}{10000}\,.
\end{split}\]

As for $B=0$, in Figure \ref{Fig::5}, we represent  function $\tilde{u}(t,s)$, bounded from above by $k_2(t)$ and from below by $k_1(t)$ for $B=\log(2+\sqrt{5})$.  Moreover, in Figure \ref{Fig::6}, it is plotted the same representation considering the constant values $t_0=\frac{2}{3}$ and $s_0=\frac{2}{3}$, respectively.
\begin{figure}[h]
	\centering
	\includegraphics[width=0.7\textwidth]{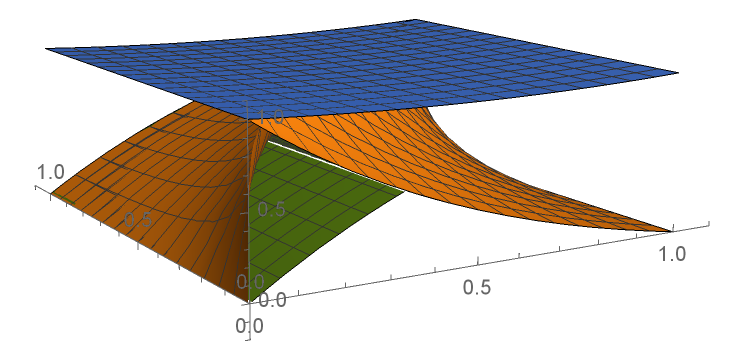}
	\caption{\scriptsize$\tilde u(t,s)$ (orange) bounded from above by $k_2(t)$ (blue) and from below by $k_1(t)$ (green).\label{Fig::5}}
\end{figure}

\begin{figure}[h]
	\centering
	\includegraphics[width=0.4\textwidth]{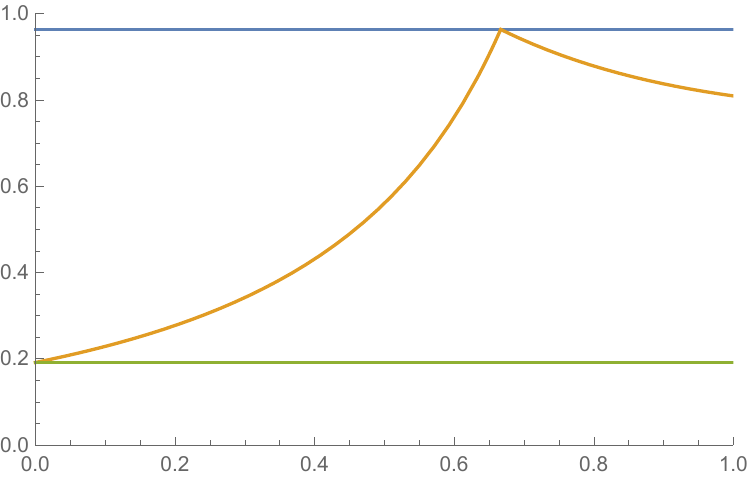} \hspace{2cm}\includegraphics[width=0.4\textwidth]{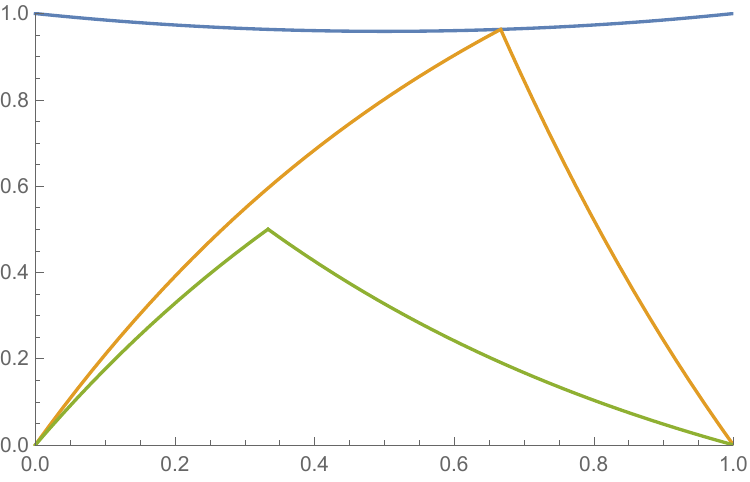} 
	\caption{\scriptsize Figure \ref{Fig::3} for $t_0=\frac{2}{3}$ on the left and for $s_0=\frac{2}{3}$ on the right.\label{Fig::6}}
\end{figure}
Thus, as in the case where $B=0$, we obtain the correspondent results of existence of solution for such a problem. First, let us write the related $(H_1)$ and $(H_2)$. Let $I=[0,1]$ and $I_1$ be defined in \eqref{Ec::I1} for $B=\log (2+\sqrt{5})$.
\begin{itemize}
	\item[$(H_1)$] There exists $p>0$ such that $f(t,u)\leqslant 6\,p$ for all $t\in I$ and $u\in[0,p]$.
		\item[$(H_2)$] There exists $q>0$ such that $f(t,u)\geqslant \dfrac{200000}{3587}u>\dfrac{u}{\frac{1}{2}\int_{a_1}^{b_1}k_1(s)\,\varPhi(s)\,ds}$ for all $t\in I_1 $ and $u\in\left[ \dfrac{q}{4},q\right] $.
\end{itemize}

Finally, we can rewrite Theorems \ref{T::5} and \ref{T::6} as follows.

\begin{theorem}\label{T::5B}
	Let $I=[0,1]$ and $I_1$ be defined in \eqref{Ec::I1} for $B=\log (2+\sqrt{5})$ and suppose that there exist positive numbers $p$, $q$ and $r$ such that $0<p<q<r$, and assume that function $f$ satisfies the following conditions:
	\begin{enumerate}[$(i)$]
		\item $f(t,u)\geqslant\dfrac{40000}{957}u>\dfrac{u}{\frac{1}{4}\int_{a_1}^{b_1}\varPhi(s)\,ds}$ for all $t\in I_1 $ and $u\in\left[ r,4\,r\right] $,
		\item $f(t,u)\leqslant 6\,q$ for all $t\in I$ and $u\in\left[ 0,4\,q\right] $, being the inequality strict for $u=q$,
		\item $f(t,u\geqslant\dfrac{200000}{3587}u$ for all $t\in I_1$ and $u\in\left[ \dfrac{p}{4},p\right] $.
	\end{enumerate}
	
	Then, for $B=\log(2+\sqrt{5})$,  problem \eqref{Ec::2ord} has at least two positive solutions, $u_1$ and $u_2$, such that $p<\|u_1\|_\infty$, $\max_{t\in I_1}u_1(t)<q<\max_{t\in I_1}u_2(t)$ and $\min_{t\in I_1}u_2(t)<r$.
\end{theorem}

\begin{theorem}\label{T::6B}
	Let $I=[0,1]$ and $I_1$ be defined in \eqref{Ec::I1} for $B=\log (2+\sqrt{5})$ and suppose $p$, $q$ and $r$ are positive numbers such that 
	\[0<p<q<4\,q\leqslant r\,,\]
	such the function $f$ satisfies the following conditions:
	\begin{enumerate}[$(a)$]
		\item $f(t,u)\leqslant 6\,r$ for all $t\in I$ and $u\in\left[ 0,r\right] $,
		\item $f(t,u)<6\,p$ for all $t\in I$ and $u\in\left[ 0,p\right] $,
		\item $f(t,u)\geqslant\dfrac{40000}{957}u$ for all $s\in I_1 $ and $u\in\left[ q,4\,q\right] $.
	\end{enumerate}
	
	Then,  for $B=\log(2+\sqrt{5})$, problem \eqref{Ec::2ord} has at least three positive solutions, $u_1$, $u_2$, $u_3\in \{u\in\mathcal P\ \mid\ \|u\|_\infty\leqslant r\}$, such that  $\max_{t\in I_1}u_1(t)<p$, $q<\min_{t\in I_1}u_2(t)$ and $p<\max_{t\in I_1}u_3(t)$ with $\min_{t\in I_1}u_3(t)<q$.
\end{theorem}

\vspace{0.5cm}

Finally, repeating the same arguments for $B<0$, we obtain that for $B\in[-2,0]$:
\[\begin{split}
k_2(t)&=\dfrac{\left( 1-e^{B\,t}\right) \,\left( e^{-B}-e^{-B\,t}\right) }{t\,(1-t)\,B\,(1-e^{-B})}\leqslant 1=K_2\,,\\\\
k_1(t)&=\begin{cases}
\dfrac{e^{-B\,t}-1}{e^{-B}-1}\,,&0\leqslant t\leqslant t_2=-\dfrac{\log\left( \frac{e^{-B}+1}{2}\right) }{B}\,,\\\\
\dfrac{e^{-B}-e^{-B\,t}}{e^{-B}-1}\,,&t_2<t\leqslant1\,,\end{cases}
\end{split}\]
and $K_1=\max_{t\in[0,1]}k_1(t)=k_1(t_2)=\dfrac{1}{2}$.

In this case, if we choose \begin{equation}
\label{Ec::I1_2}
I_1=[a_1,b_1]=\left[ -\dfrac{\log(\frac{e^{-B}+3}{4})}{B},-\dfrac{\log(\frac{3\,e^{-B}+1}{4})}{B}\right]\,,\end{equation} it is verified that $m_1=k_1(a_1)=k_1(b_1)=\dfrac{1}{4}$.

\begin{remark}\label{R::I12}
	For $B\in[-2,0]$, both $a_1$ and $b_1$ are decreasing functions of $B$. Moreover,
	\[	a_1\in\left[\frac{1}{4},\frac{1}{2} \log \left(\frac{1}{4} \left(3+e^2\right)\right)\right] \subset\left[\frac{1}{4},\frac{12}{25}\right]\,,\]
	and
	\[	b_1\in\left[\frac{3}{4},\frac{1}{2} \log \left(\frac{1}{4} \left(1+3 e^2\right)\right)\right]\subset\left[\frac{3}{4},\frac{22}{25}\right] \,.\]
\end{remark}

Moreover, for the choice $B=\log(\sqrt{5}-2)\in[-2,0]$ and he correspondent interval $I_1$ defined in \eqref{Ec::I1_2}, we attain the same bounds as for $B=\log(2+\sqrt{5})$. Thus, $(H_1)$ and $(H_2)$ coincide in both cases and Theorems \ref{T::5B} and \ref{T::6B} remain valid for $B=\log(\sqrt{5}-2)$.

It is important to recall that for all $B\in[-2,2]$ any suitable bounds can be obtained by using the expression of $k_1$ and $k_2$ without any additional difficulty.

\vspace{0.5cm}

If $B\notin[-2,2]$, the study is much more complicated. However the approach can also be done. For instance let us choose $B=-2\,\pi$, we have
\[-g(t,s)=\begin{cases}
\dfrac{\left(1- e^{-2\,\pi\,s}\right) \left( e^{2\,\pi}-e^{2\pi\,t}\right) }{2\,\pi\,(e^{2\pi}-1)}\,,&0\leqslant s\leqslant t\leqslant 1\,,\\\\
\dfrac{\left( e^{2\,\pi\,(1-s)}-1\right) \left( e^{2\,\pi\,t}-1\right) }{2\,\pi\,(e^{2\,\pi}-1)}\,,&0<t<s\leqslant 1\,.
\end{cases}\]

%
%
%

In this case, we are not able to give the exact expression  of $k_1(t)$, but we can obtain a lower bound for $\tilde u(t,s)$ as follows: 

%
\[k_1(t)=\begin{cases}
\dfrac{e^{2\,\pi\,t}-1}{e^{2\,\pi}-1}\,,&0\leqslant t\leqslant t_3\,,\\\\
\dfrac{250000}{62037}\left( 1-e^{\frac{-451}{500}\,\pi}\right)\,\dfrac{e^{2\,\pi}-e^{2\,\pi\,t}}{2\,\pi\,(e^{2\,\pi}-1)}\,,&t_3<t\leqslant 1\,,
\end{cases}\]
where $t_3=\dfrac{1}{2\,\pi}\log\left( \dfrac{125000\left( e^{\frac{549}{500}\,\pi}-e^{2\,\pi}\right) -62037\,\pi}{125000\left( e^{-\frac{451}{500}\,\pi}-1\right) -62037\,\pi}\right) \approxeq0.844992$.

Moreover, $\tilde u(t,s)\leqslant 1$ for all $(t,s)\in [0,1]\times [0,1]$.

In this case, we conclude that \[K_1=\max_{t\in[0,1]}k_1(t)=k_1(t_3)\approxeq 0.37642>\dfrac{47}{125}\,.\]

As in previous cases, in Figure \ref{Fig::7}, we represent the function $\tilde{u}(t,s)$, bounded from above by $K_2=1$ and from below by $k_1(t)$.  Moreover, in Figure \ref{Fig::8}, it is plotted the same representation considering the constant values $t_0=\frac{21}{25}$ and $s_0=\frac{21}{25}$, respectively.

\begin{figure}[h]
	\centering
	\includegraphics[width=0.7\textwidth]{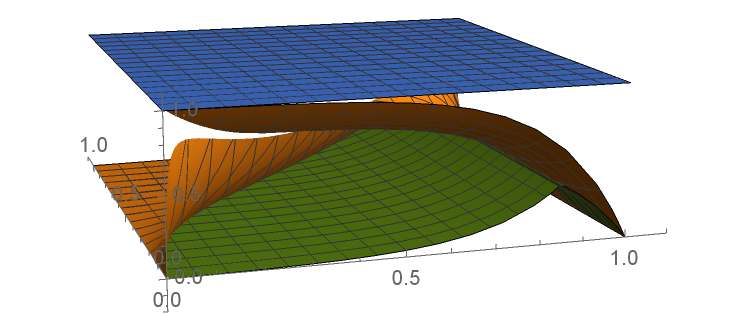}
	\caption{\scriptsize$\tilde u(t,s)$ (orange) bounded from above by $k_2(t)$ (blue) and from below by $k_1(t)$ (green).\label{Fig::7}}
\end{figure}

\begin{figure}[h]
	\centering
	\includegraphics[width=0.4\textwidth]{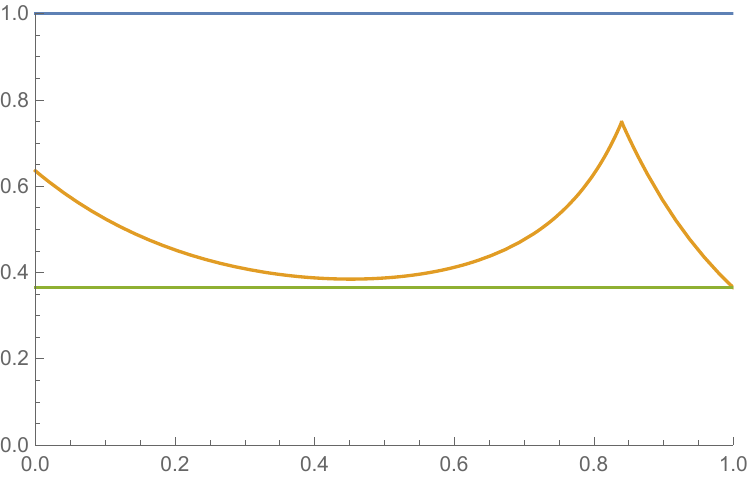} \hspace{2cm}\includegraphics[width=0.4\textwidth]{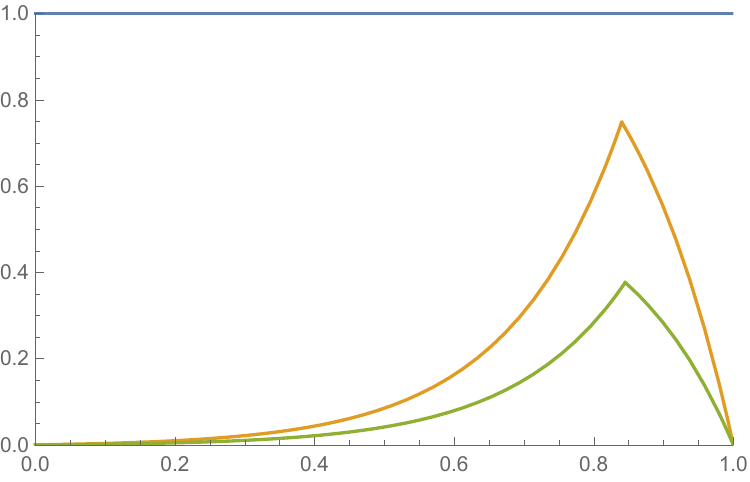} 
	\caption{\scriptsize Figure \ref{Fig::7} for $t_0=\frac{21}{25}$ on the left and for $s_0=\frac{21}{25}$ on the right.\label{Fig::8}}
\end{figure}

If we choose, \begin{equation}
\label{Ec::I1_3} I_1=[a_1,b_1]=\left[ \dfrac{\log\left(\frac{3+e^{2\,\pi}}{4}\right) }{2\,\pi},0.9151\right] \subset \left[0.78 ,0.9151\right] \,,
\end{equation}
we have that $m_1=k_1(a_1)=\dfrac{1}{4}$ and
\[\begin{split}
\int_{a_1}^{b_1}\varPhi(s)\,ds&\approxeq 0.0172072>\dfrac{43}{2500}\,,\\
\int_{a_1}^{b_1}k_1(s)\,\varPhi(s)\,ds&\approxeq 0.005393>\dfrac{539}{100000}\,.
\end{split}\]

	Thus, as in the previous cases, we deduce the correspondent results of existence of solution for this problem. Let $I=[0,1]$ and $I_1$ be defined in \eqref{Ec::I1_3}. First, let us write the related $(H_1)$ and $(H_2)$ as follows:
	\begin{itemize}
		\item[$(H_1)$] There exists $p>0$ such that $f(t,u)\leqslant 6\,p$ for all $t\in I$ and $u\in[0,p]$.		
		\item[$(H_2)$] There exists $q>0$ such that $f(t,u)\geqslant \dfrac{12500000\,u}{25333}>\dfrac{u}{\frac{47}{125}\int_{a_1}^{b_1}k_1(s)\,\varPhi(s)\,ds}$ for all $t\in I_1 $ and $u\in\left[ \dfrac{q}{4},q\right] $.
	\end{itemize}
	
	So, we can rewrite Theorems \ref{T::5} and \ref{T::6} as follows.
	
	\begin{theorem}\label{T::5.7}
	 Let $I=[0,1]$ and $I_1$ be defined in \eqref{Ec::I1_3}  and suppose that there exist positive numbers $p$, $q$ and $r$ such that $0<p<q<r$, and suppose that function $f$ satisfies the following conditions:
		\begin{enumerate}[$(i)$]
			\item $f(t,u)\geqslant\dfrac{10000}{43}u>\dfrac{u}{\frac{1}{4}\int_{a_1}^{b_1}\varPhi(s)\,ds}$ for all $t\in I_1$ and $u\in\left[ r,4\,r\right] $,
			\item $f(t,u)\leqslant6\,q$ for all $t\in I$ and $u\in\left[ 0,4\,q\right] $, being the inequality strict for $u=q$,
			\item $f(t,u)\geqslant\dfrac{12500000}{25333}u$ for all $t\in I_1$ and $u\in\left[ \dfrac{p}{4},p\right] $.
		\end{enumerate}
		
		Then, for $B=-2\,\pi$,  problem \eqref{Ec::2ord} has at least two positive solutions, $u_1$ and $u_2$, such that $p<\|u_1\|_\infty$, $\max_{t\in I_1}u_1(t)<q<\max_{t\in I_1}u_2(t)$ and $\min_{t\in I_1}u_2(t)<r$.
	\end{theorem}
	
	Let us consider the following continuous function
	\begin{equation}\label{F::1}f(t,u)=\begin{cases}
	\left(\frac{1007}{88}+\frac{225}{88}\,t\right) \dfrac{50000000}{1190651}u\,,&u\leqslant \dfrac{1}{28}\,,\\\\
		\left(\frac{1007}{88}+\frac{225}{88}\,t\right)\dfrac{3125000}{58341899\,u}\,,&\dfrac{1}{28}<u\leqslant 14\,,\\\\
	\left(\frac{1007}{88}+\frac{225}{88}\,t\right)\dfrac{3125000}{58341899\,u}+\dfrac{10000}{43}(u-14)\,u\,,& u>14	\,.
	\end{cases}\end{equation}
	
	It is easy to verify that this function $f$ satisfies the hypotheses of Theorem \ref{T::5.7}.
	
	\begin{enumerate}[$(i)$]
		\item  For the construction of $f$, $(i)$ is trivially fulfilled for all $u\geqslant 15$.
		\item If $q=\dfrac{7}{2}$, then $f(t,u)\leqslant f\left( t,\dfrac{1}{28}\right) =	\left(\frac{1007}{88}+\frac{225}{88}\,t\right)\dfrac{50000000}{1190651}\cdot\dfrac{1}{28}\leqslant \dfrac{25000000}{1190651}< \dfrac{7}{2}\cdot 6$ for all $t\in I $ and $u \in [0,14]$.
		\item If $p=\dfrac{1}{28}$, clearly $f(t,u)\geqslant\dfrac{47}{4}\cdot\dfrac{50000000}{1190651}u=\dfrac{12500000}{25333}\,u$, for all $u\leqslant p$ and $t\in I_1$ for all $t\in \left[ \dfrac{3}{25},1\right]$, in particular for $t\in I_1$ and $u\in\left[ \dfrac{1}{112},\dfrac{1}{28}\right]$.
	\end{enumerate}

\begin{remark}\label{R::2}
	Realize that, from Remarks \ref{R::I11} and \ref{R::I12}, the conditions imposed to $f$ are stronger than those imposed in Theorems \ref{T::5.3} and \ref{T::5B}, then the obtained $f$ is also valid for these results.
\end{remark}

Thus, we can conclude that for $B\in\{-2\,\pi,\log (\sqrt{5}-2),0,\log(\sqrt{5}+2)\}$ problem \eqref{Ec::2ord},   with $f$ defined in \eqref{F::1}, has at least two positive solutions, $u_1$ and $u_2$, such that $\dfrac{1}{28}<\|u_1\|_\infty$, $\max_{t\in I_1}u_1(t)<\dfrac{7}{2}<\max_{t\in I_1}u_2(t)$, $\min_{t\in I_1}u_2(t)<15$.
	\begin{theorem}\label{T::5.8}
		 Let $I=[0,1]$ and $I_1$ be defined in \eqref{Ec::I1_3}  and suppose $p$, $q$ and $r$ are positive numbers such that 
		\[0<p<q<4\,q\leqslant r\,,\]
		such the function $f$ satisfies the following conditions:
		\begin{enumerate}[$(a)$]
			\item $f(t,u)\leqslant 6\,r$ for all $t\in I$ and $u\in\left[ 0,r\right] $,
			\item $f(t,u)<6\,p$ for all $s\in I$ and $u\in\left[ 0,p\right] $,
			\item $f(t,u)\geqslant\dfrac{10000}{43}u$ for all $s\in I_1 $ and $u\in\left[ q,4\,q\right] $.
		\end{enumerate}
		
		Then,  for $B=-2\,\pi$, problem \eqref{Ec::2ord} has at least three positive solutions, $u_1$, $u_2$, $u_3\in \{u\in\mathcal P\ \mid\ \|u\|_\infty\leqslant r\}$, such that  $\max_{t\in I_1}u_1(t)<p$, $q<\min_{t\in I_1}u_2(t)$ and $p<\max_{t\in I_1}u_3(t)$ with $\min_{t\in I_1}u_3(t)<q$.
	\end{theorem}

Let us consider the following continuous function
\begin{equation}\label{F::2} f(t,u)=\begin{cases}
12\,\left( \frac{31}{28}t+\frac{25}{28}\right) \,u^3\,,&u\leqslant 16\,,\\\\
49152\left( \frac{31}{28}t+\frac{25}{28}\right)\,,&u>16\,.
\end{cases}\end{equation}

It is easy to verify that $f$ fulfills the hypotheses of Theorem \ref{T::5.8}:
\begin{enumerate}[$(a)$]
	\item If $r=16384$, then $f(t,u)\leqslant 98304=6\cdot16384$ for all $t\in I$ and $u\in [0,16384]$,
	\item If $p=\dfrac{1}{2}$, then $f(t,u)\leqslant12\cdot 2\cdot \dfrac{1}{4}\cdot{1}{2}=6\cdot\dfrac{1}{2}$, for all $t\in I$ and $u\in\left[ 0,\dfrac{1}{2}\right]$, 
	\item If $q=4$, then $f(t,u)\geqslant 12\cdot\dfrac{5}{4}\,16\, u=240\,u>\dfrac{10000}{43}\,u$ for all $t\in \left[ \dfrac{3}{25},1\right] $, in particular for $t\in I_1$ and $u\in[4,16]$.
\end{enumerate}

As in Remark \ref{R::2}, the imposed conditions on $f$ are stronger than those in Theorems \ref{T::5.3} and \ref{T::6B}. Thus, we can conclude that for $B\in\{-2\,\pi,\log (\sqrt{5}-2),0,\log(\sqrt{5}+2)\}$ problem \eqref{Ec::2ord},   with $f$ defined in \eqref{F::2}, has at least three positive solutions, $u_1$, $u_2$, $u_3\in \{u\in\mathcal P\ \mid\ \|u\|_\infty\leqslant 16384\}$, such that  $\max_{t\in I_1}u_1(t)<\dfrac{1}{2}$, $4<\min_{t\in I_1}u_2(t)$ and $\dfrac{1}{2}<\max_{t\in I_1}u_3(t)$ with $\min_{t\in I_1}u_3(t)<4$.
	\begin{itemize}
	\item Fourth order: $n=4$ and $k=2$.
\end{itemize}

In \cite{CabSaa} it is proved that the related Green's function of operator $T_4[M]\,u(t)=u^{(4)}(t)+M\,u(t)$ on $X_2$,  with $I=[0,1]$, satisfies the property $(Pg_1)$ if, and only if, $M\in (-\lambda_1^4,\lambda_2^4]$, where $\lambda_1\approxeq 4.73$ is the least positive solution of
\[\cos (\lambda ) \cosh (\lambda )=1\,,\]
and $\lambda_2\approxeq 5.55$ is the least positive solution of
\[\tan\left( \dfrac{\lambda}{\sqrt{2}}\right) =\tanh\left( \dfrac{\lambda}{\sqrt{2}}\right)\,.\]

 In particular, it is fulfilled for $M=0$. Thus, let us study the following problem:
\begin{equation}\label{Ec::Ex1}
u^{(4)}(t)=f(t,u(t))\,,\quad t\in[0,1]\,,\quad
u(0)=u'(0)=u(1)=u'(1)=0\,.
\end{equation}

The related Green's function (see \cite{CaCiMa}) is given by
\[g(t,s)=\begin{cases}
\dfrac{s^2}{6}\,(1-t)^2\,(3\,t-s-2\,s\,t)\,,&0\leqslant s\leqslant t\leqslant 1\,,\\\\
\dfrac{(1-s)^2}{6}\,t^2\,(3\,s-t-2\,s\,t)\,,&0<t<s\leqslant 1\,.
\end{cases}\]

Thus, after tedious calculus, we can ensure that
%
%
%
\[k_1(t)\begin{cases}
\dfrac{t^2}{2}\,(1-t)\,,&0\leqslant t\leqslant\dfrac{1}{2}\,,\\\\
\dfrac{t}{2}\,(1-t)^2,&\dfrac{1}{2}<t\leqslant 1\,,
\end{cases}\]
and
\[k_2(t)=\begin{cases}
\dfrac{t}{2}\,(1-t)^2\,,&0\leqslant t\leqslant\dfrac{1}{4}\,,\\\\
\dfrac{1-t}{24}\,(1+2\,t)^2,&\dfrac{1}{4}<t\leqslant\dfrac{1}{2}\,,\\\\
\dfrac{t}{24}\,(3-2\,t)^2\,,&\dfrac{1}{2}< t\leqslant\dfrac{3}{4}\,,\\\\
\dfrac{t^2}{2}\,(1-t),&\dfrac{3}{4}<t\leqslant 1\,.
\end{cases}\]

Hence, we have \[K_2=\max_{t\in[0,1]}k_2(t)=k_2\left( \dfrac{1}{2}\right) =\dfrac{1}{12} \ \text{ and } K_1=\max_{t\in[0,1]}k_1(t)=k_1\left( \dfrac{1}{2}\right) =\dfrac{1}{16}\,.\]

As in the second order case, in Figure \ref{Fig::3}, we represent the function $\tilde{u}(t,s)$, bounded from above by $k_2(t)$ and from below by $k_1(t)$.  Moreover, in Figure \ref{Fig::4}, it is shown the same representation considering the constant values $t_0=\frac{3}{4}$ and $s_0=\frac{3}{4}$, respectively.

\begin{figure}[h]
	\centering
	\includegraphics[width=0.7\textwidth]{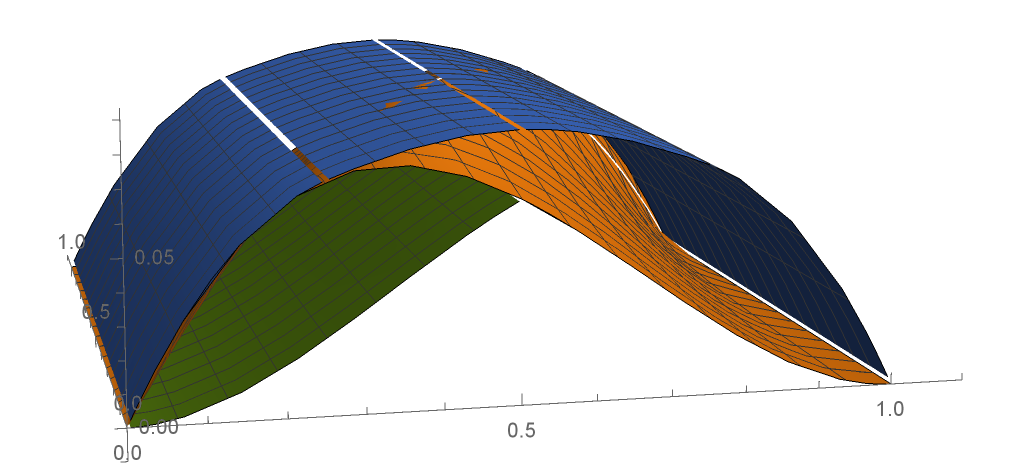}
	\caption{\scriptsize$\tilde u(t,s)$ (orange) bounded from above by $k_2(t)$ (blue) and from below by $k_1(t)$ (green).\label{Fig::3}}
\end{figure}

\begin{figure}[h]
	\centering
	\includegraphics[width=0.4\textwidth]{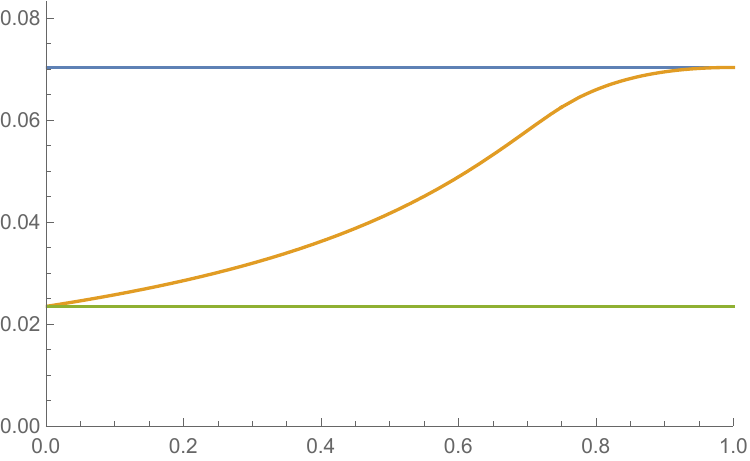} \hspace{2cm}\includegraphics[width=0.4\textwidth]{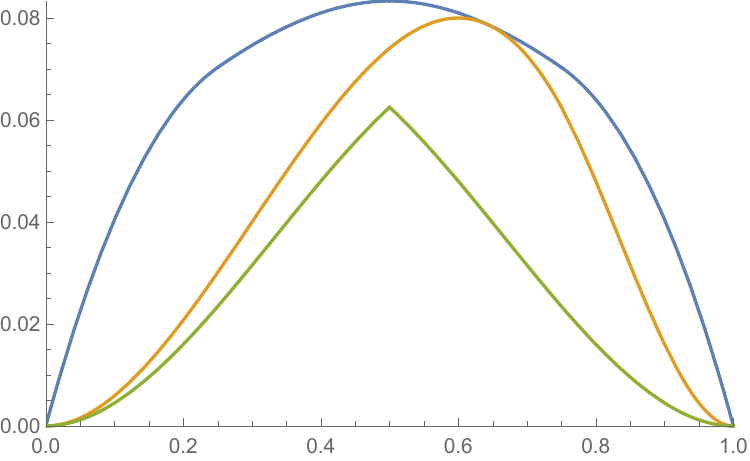} 
	\caption{\scriptsize Figure \ref{Fig::3} for $t_0=\frac{3}{4}$ on the left and for $s_0=\frac{3}{4}$ on the right.\label{Fig::4}}
\end{figure}

Let us choose $I_1=[a_1,b_1]=\left[ \dfrac{1}{3},\dfrac{2}{3}\right]$, in this case we have $m_1=\min_{t\in I_1}k_1(t)=k_1\left( \dfrac{1}{3}\right) =\dfrac{1}{27}$ and
\[\begin{split}
\int_0^1\varPhi(s)\,ds&=\dfrac{1}{30}\,,\\
\int_{a_1}^{b_1}\varPhi(s)\,ds&=\dfrac{47}{2430}\,,\\
\int_{a_1}^{b_1}k_1(s)\,\varPhi(s)\,ds&=\dfrac{462461}{470292480}\,.\end{split}\]

Hence, as in the  second order case, we can obtain the correspondent $(H_1)$ and $(H_2)$:
	\begin{itemize}
	\item[$(H_1)$] There exists $p>0$ such that $f(t,u)\leqslant 360\,p$ for all $t\in [0,1]$ and $u\in[0,p]$.		
	\item[$(H_2)$] There exists $q>0$ such that $f(t,u)\geqslant \dfrac{627056640\,u}{426461}$ for all $t\in\left[\dfrac{1}{3},\dfrac{2}{3}\right] $ and $u\in\left[ \dfrac{4}{9}q,q\right] $.
\end{itemize}

Finally, we can rewrite Theorems \ref{T::5} and \ref{T::6} as follows.

\begin{theorem}
	Suppose that there exist positive numbers $p$, $q$ and $r$ such that $0<p<q<r$, and suppose that function $f$ satisfies the following conditions:
	\begin{enumerate}[$(i)$]
		\item $f(t,u)\geqslant\dfrac{65610}{47}u$ for all $t\in\left[\dfrac{1}{3},\dfrac{2}{3}\right]  $ and $u\in\left[ r,\dfrac{9}{4}r\right] $, being the inequality strict for $u=r$,
		\item $f(t,u)<360\,q$ for all $t\in[0,1]$ and $u\in\left[ 0,\dfrac{9}{4}q\right] $, being the inequality strict for $u=q$,
		\item $f(t,u)>\dfrac{627056640}{426461}u$ for all $t\in\left[\dfrac{1}{3},\dfrac{2}{3}\right]  $ and $u\in\left[ \dfrac{4}{9}p,p\right] $.
	\end{enumerate}
	
	Then  problem \eqref{Ec::Ex1} has at least two positive solutions, $u_1$ and $u_2$, such that $p<\|u_1\|_\infty$, $\max_{t\in I_1}u_1(t)<q<\max_{t\in I_1}u_2(t)$ and $\min_{t\in I_1}u_2(t)<r$.
\end{theorem}

Let us choose a particular function $f$ that the satisfies the conditions imposed in previous result.
\[f(t,u)=\begin{cases}
1296\,t\,,&u\leqslant \dfrac{1}{36}\,,\\\\
\dfrac{t}{u^2}\,,&\dfrac{1}{36}<u\leqslant \dfrac{33}{4}\,,\\\\
\dfrac{64\,t}{35937}\left( u-\dfrac{29}{4}\right) ^5u\,,&u>\dfrac{33}{4}\,.
\end{cases}\]

Let us choose $p=\dfrac{1}{16}$, $q=\dfrac{11}{3}$ and $r=27$, we have:
	\begin{enumerate}[$(i)$]
	\item For all $t\in \left[\dfrac{1}{3},\dfrac{2}{3}\right]  $ and $u\in\left[ 27,\dfrac{243}{4}\right] $, $f(t,u)=\dfrac{64\,t}{35937}\left( u-\dfrac{29}{4}\right)^5u\geqslant\dfrac{64}{3\,\cdot\,35937}\,\left( \dfrac{79}{4}\right) ^5\,u=\dfrac{3077056399\,u}{1724976}>\dfrac{65610\,u}{47}$.
	\item For all $t\in[0,1]$ and $u\in\left[ 0,\dfrac{33}{4}\right] $, we have $f(t,u)\leqslant \dfrac{1}{\left( 1/36\right) ^2}=1296<\dfrac{11}{3}\,\cdot\,360$. 
	\item For all $t\in\left[\dfrac{1}{3},\dfrac{2}{3}\right]  $ and $u\in\left[ \dfrac{1}{36},\dfrac{1}{16}\right] $, $f(t,u)=\dfrac{t}{u^3}\,u\geqslant \dfrac{1/3}{(1/16)^3}\,u=\dfrac{4096}{3}\,u>\dfrac{627056640\,u}{462461}$.
	\end{enumerate}

	Hence, for this function the problem \eqref{Ec::Ex1} has at least two positive solutions, $u_1$ and $u_2$, such that $\dfrac{1}{16}<\|u_1\|_\infty$, $\max_{t\in I_1}u_1(t)<\dfrac{11}{3}<\max_{t\in I_1}u_2(t)$ and $\min_{t\in I_1}u_2(t)<27$.
	
	Now, Theorem \ref{T::6} reads as follows:
\begin{theorem}\label{T::6-4}
	Suppose $p$, $q$ and $r$ are positive numbers for which the following inequalities are fulfilled:
	\[0<p<q<\dfrac{9}{4}q\leqslant r\,,\]
	and that the function $f$ satisfies the following conditions:
	\begin{enumerate}[$(a)$]
		\item $f(t,u)\leqslant 360\,r$ for all $t\in [0,1]$ and $u\in\left[ 0,r\right] $,
		\item $f(t,u)<360\,p$ for all $s\in [0,1]$ and $u\in\left[ 0,p\right] $,
		\item $f(t,u)\geqslant\dfrac{65610}{47}u$ for all $s\in\left[\dfrac{1}{3},\dfrac{2}{3}\right]  $ and $u\in\left[ q,\dfrac{9}{4}q\right] $, being the inequality strict for $u=q$.
	\end{enumerate}
	
	Then problem \eqref{Ec::Ex1} has at least three positive solutions, $u_1$, $u_2$, $u_3\in \{u\in\mathcal P\ \mid\ \|u\|_\infty\leqslant r\}$, such that  $\max_{t\in I_1}u_1(t)<p$, $q<\min_{t\in I_1}u_2(t)$ and $p<\max_{t\in I_1}u_3(t)$ with $\min_{t\in I_1}u_3(t)<q$.
\end{theorem}

Let us see that the following function satisfies these hypotheses:
\[f(t,u)=\begin{cases}
(2+3\,t)\,u^2\,,&u\leqslant \dfrac{1}{2}\,,\\\\
\left( u-\dfrac{1}{2}\right) u^4+(2+3\,t)\,u^2\,,&\dfrac{1}{2}<u\leqslant 14\,,\\\\
519204+588\,t\,,&u>14\,.
\end{cases}\]

Let us choose $p=\dfrac{1}{2}$, $q=\dfrac{56}{9}$ and $r=1444$, we can check all the hypotheses of Theorem \ref{T::6-4}.
	\begin{enumerate}[$(a)$]
	\item  For all $t\in [0,1]$ and $u\in\left[ 0,1444\right] $, $f(t,u)\leqslant 519204+588=519792<360\,\cdot\,1444=519840\,r$.
	\item For all $s\in [0,1]$ and $u\in\left[ 0,\dfrac{1}{2}\right] $, then $f(t,u)=(2+3\,t)\,u^2\leqslant \dfrac{5}{4}<360\,\cdot\,\dfrac{1}{2}=180$.
	\item For all $s\in \left[\dfrac{1}{3},\dfrac{2}{3}\right] $ and $u\in\left[ \dfrac{56}{9},14\right] $, $f(t,u)=\left( u-\dfrac{1}{2}\right) u^4+(2+3\,t)\,u^2=\left( \left( u-\dfrac{1}{2}\right) u^3+(2+3\,t)\,u\right) \,u\geqslant\left( \left( \dfrac{56}{9}-\dfrac{1}{2}\right) \left( \dfrac{56}{9}\right) ^3+3\,\dfrac{56}{9}\right) =\dfrac{9166696\,u}{6561}>\dfrac{65610\,u}{47}$.
\end{enumerate}

Hence, for such a $f$ the problem \eqref{Ec::Ex1} has at least three positive solutions, $u_1$, $u_2$ and $u_3$, such that $\max_{t\in I_1}u_1(t)<\dfrac{1}{2}$, $\dfrac{56}{9}<\min_{t\in I_1}u_2(t)$ and $\dfrac{1}{2}<\max_{t\in I_1}u_3(t)$ with $\min_{t\in I_1}u_3(t)<\dfrac{56}{9}$.

	\end{document}